\documentclass{amsart}
\usepackage{amssymb,latexsym}
\usepackage{amsfonts}
\usepackage[mathscr]{eucal}
\usepackage{enumerate}
\usepackage[colorlinks,linkcolor=black,anchorcolor=black, citecolor=black]{hyperref}
\usepackage{fancyhdr}
\usepackage[deletedmarkup=xout,final]{changes}
	\definechangesauthor[name={Per cusse}, color=orange]{per}
\setremarkmarkup{(#2)}

\newtheorem{thm}{Theorem}[section]
\newtheorem{yl}[thm]{Lemma}

\theoremstyle{remark}
\newtheorem{zj}{Remark}[section]
\theoremstyle{definition}
\newtheorem{dy}{Definition}[section]
\numberwithin{equation}{section}
\allowdisplaybreaks[4]

\newcommand\relphantom[1]{\mathrel{\phantom{#1}}}
\newcommand\repl{\relphantom}
\usepackage[utf8]{inputenc}

\begin{document}
\title{stability of line bundle mean curvature flow}
\author{Xiaoli Han}
\address{Xiaoli Han\\ Math department of Tsinghua university\\ Beijing\\ 100084\\ China\\} \email{hanxiaoli@mail.tsinghua.edu.cn}
\author{Xishen Jin}
\address{Xishen Jin\\ School of Mathematics\\ Remin University of China\\ Beijing\\ 100872\\ China\\} \email{jinxsh@mail.ruc.edu.cn}

\begin{abstract}
Let $(X,\omega)$ be a \deleted{connected} compact K\"ahler manifold of complex dimension $n$ and $(L,h)$ be a holomorphic line bundle over $X$. The line bundle mean curvature flow was introduced by Jacob-Yau \added{in order to find deformed Hermitian-Yang-Mills metrics on $L$}. In this paper, we consider the stability of the line bundle mean curvature flow. Suppose there exists a deformed Hermitian Yang-Mills metric $\hat h$ on $L$. We prove that the line bundle mean curvature flow converges to $\hat h$\added{ exponentially} in $C^\infty$ sense as long as the initial metric is close to $\hat h$ in $C^2$-norm.
\end{abstract}
\maketitle
\section{Introduction}

Let $(X,\omega)$ be a \deleted{connected} compact K\"ahler manifold of complex dimension $n$ and $L$ be a holomorphic line bundle over $X$. Given a\added{ Hermitian} metric $h$ on $L$, we define \replaced{a}{the} complex function \added{$\zeta$ on $X$ by}
\[
\zeta:=\frac{(\omega-F_h)^n}{\omega^n}\deleted{,}
\]
where $F_h=-\partial\bar \partial \log h$ is the curvature of the Chern connection with respect to the metric $h$. It is easy to see that the average of this function is a fixed complex number
\[
Z_{L,[\omega]}:= \int_X \zeta \frac{\omega^n}{n!}
\]
depending only \added{on} the cohomology classes \deleted{on} $c_1(L)$ and $[\omega]\in H^{1,1}(X,\mathbb{R})$. Let $\hat \theta$ be the argument of $Z_{L,[\omega]}$. \deleted{Suppose $Z_{L,[\omega]}$ does not vanish, we choose $\hat\theta$ such that $\mathop{Im}( e^{-\sqrt{-1}\hat\theta}Z_{L,[\omega]})=0$.}

\begin{dy}
  A Hermitian metric $h$ on $L$ is said to be\added{ a} deformed Hermitian-Yang-Mills (dHYM) metric if it satisfies
  \begin{equation}
    \label{eqn-dHYM-form}
    \operatorname{\mathop{Im}}(\omega-F)^n=\tan(\hat \theta) \operatorname{\mathop{Re}}(\omega-F)^n.
  \end{equation}
\end{dy}

\deleted{Suppose $\lambda_j$ are the eigenvalues of $\omega^{-1}F\added{\in \mathrm{End}(T^{1,0}X)}$.} \replaced{We define the Lagrangian Phase operator $\theta: \wedge^{1,1}X \to \mathbb{R}$ by}{Define}

\[
\theta(F)=\sum_{j=1}^n\arctan\lambda_j
\]
\added{where $\lambda_j$($j=1,\cdots,n$) are the eigenvalues of $\omega^{-1}F \in \mathrm{End}(T^{1,0}X)$.} Then \added{according to arguments in \cite{JY},} the equation (\ref{eqn-dHYM-form}) \replaced{is equivalent to}{can rewrite  as}
\begin{equation}\label{eqn-dHYM-angle}
\theta(F)=\hat\theta(\text{  mod  } 2\pi).
\end{equation}
\replaced{As discussed in \cite{CXY}, we}{we} also remark \deleted{in \cite{CXY}}\added{that} the constant $\hat \theta$ in \eqref{eqn-dHYM-angle} can be obtained by considering the ``winding angle'' of
\[
\gamma(t) =\int_X e^{-t\sqrt{-1} \omega} Ch(L)
\]
as $t$ runs from $+\infty$ to $1$ if $\gamma(t)$ do\added{es} not cross $0\in \mathbb{C}$.

\added{According to\added{ the} superstring theory, the spacetime of the universe is constrained to be the product of a compact Calabi-Yau threefold and a Lorentzian manifold of four dimension. A `duality' relates the geometry of one Calabi-Yau manifold to another `mirror' Calabi-Yau manifold. }The dHYM equation was \added{first} discovered by Marino et all \cite{MMMS} as the requirement for a D-brane on the B-model of mirror symmetry to be supersymmetric. This is explained by \replaced{Leung-Yau-Zaslow}{Leung, Yau and Zaslow}\cite{LYZ} in mathematical language.\deleted{ Recently, it has been studied \added{actively(e.g. \cite{CJY}, \cite{CXY}, \cite{CY}, \cite{HY}, \cite{HJ}, \cite{JY}, \cite{Ping}, \cite{SS}, \cite{T2019} etc)}\deleted{ by Jacob-Yau \cite{JY}, Collins-Jacob-Yau \cite{CJY}, Collins-Xie-Yau \cite{CXY} and some other people}. According to\added{ the} superstring theory, the spacetime of the universe is constrained to be a product of a compact Calabi-Yau threefold and a four-dimensional Lorentzian manifold. A `duality' relates the geometry of one Calabi-Yau manifold to another `mirror' Calabi-Yau manifold.} From a view\deleted{d}point of differential geometry, this might be thought of a relationship between the existence of `nice' metrics on the line bundle over one Calabi-Yau manifold and special Lagrangian submanifolds in another Calabi-Yau manifold. In \cite{LYZ}, Leung-Yau-Zaslow showed that the \replaced{dHYM}{deformed Hermitian-Yang-Mills} equation on a line bundle corresponds to the special Lagrangian equation in the mirror. Recently, the dHYM metrics have been studied \added{actively(e.g. \cite{CJY}, \cite{CXY}, \cite{CY}, \cite{HY}, \cite{HJ}, \cite{JY}, \cite{Ping}, \cite{SS}, \cite{T2019} etc)}\deleted{ by Jacob-Yau \cite{JY}, Collins-Jacob-Yau \cite{CJY}, Collins-Xie-Yau \cite{CXY} and some other people}.

In \cite{JY}, Jacob-Yau gave a necessary and sufficient condition for the existence of dHYM metrics on the K\"ahler surface. More precisely, they proved that:
\begin{thm}[\cite{JY}]
Let $L$ be a holomorphic line bundle on K\replaced{\"a}{a}hler surface $X$. Then there exists one dHYM metric on $L$ if and only if there exists \replaced{a}{one} metric $h$ such that $\Omega(h)=\cot(\hat \theta)\omega+\replaced{\sqrt{-1}}{i}F>0.$
\end{thm}

In order to study the existence of dHYM metrics on high dimensional K\"ahler manifolds, Jacob-Yau \cite{JY} introduced \replaced{a}{one} parabolic evolution flow for the metric $h$ on the line bundle $L$.

\begin{dy}(Line bundle MCF)
  Given one Hermitian metric $h_0$ on $L$, we define a flow of metrics $h\replaced{_t}{(t)}= e^{-u\replaced{_t}{(t)}}h_0$ by the following equation:
\begin{equation}
\label{mainequation}
\frac{d}{dt}u\added{_t}=\theta(F\replaced{_{h_t}}{(t)})-\hat\theta
\end{equation}
\added{where $F_{h_t}=F_{h_0}+\deleted{\sqrt{-1}}\partial \bar \partial u_t$ is the Chern curvature of $L$ with respect to $h_t$.}
\end{dy}

The flow \eqref{mainequation} can be regarded as the complex version of the mean curvature flow for the Lagrangian graph. Thus it is also called line bundle mean curvature flow (line bundle MCF). \replaced{As shown in \cite{JY}, this}{This} equation is parabolic and exists in a short time $[0,\varepsilon)$. \added{They also proved the following theorem on extension of line bundle MCF.}

\begin{thm}[\cite{JY}]
  Let $L$ be a holomorphic line bundle over the compact K\"ahler manifold $X$ and $h\replaced{_t}{(t)}$ be a path of metrics on  $L$ solving (\ref{mainequation}). Assume that $Z_{L, [\omega]}\neq 0$ and $|\nabla F\added{_t}|^2_g\leq C$ uniformly in time \added{$[0,T)$}, then all derivatives $|\nabla^k F\added{_t}|_g$ are bounded by some constant $C_k$. Furthermore, if $T$ is finite, then the flow can be extended to $T+\varepsilon$. If $T=\infty$, then the flow subsequently converges to a smooth solution of (\ref{eqn-dHYM-angle}).
\end{thm}

\added{In \cite{JY},} Jacob-Yau also studied the long time existence and convergence of the line bundle MCF under some assumptions.

\begin{thm}[\cite{JY}, Theorem 1.3]
\label{thm-JY1}
Let $(X,\omega)$ be a K\"ahler manifold with non-negative orthogonal bisectional curvature and $L$ be an ample line bundle. Then there exists a natural number $k$ such that $L^{\otimes k}$ admits a dHYM metric and it is constructed via a smoothly converging family of metrics along the line bundle MCF.
\end{thm}

\begin{zj}
  Indeed, the condition on $k$ in \added{Theorem \ref{thm-JY1}} guarantees that the initial data \added{$u_0$} satisfies the so-called hypercritical condition, i.e. $\theta(F\replaced{_{u_0}}{(0)})>\frac{(n-1)\pi}{2}$\added{. As a consequence of this hypercritical condition, $F_{u_t}>0$ for all $t\geq 0$.} \replaced{Then}{so that} the operator $\theta(F\replaced{_{u_t}}{(t)})$ is concave and the Evans-Krylov \replaced{theory}{estimate} works\added{ for higher order estimates}.
\end{zj}

In \cite{CJY}, Collins-Jacob-Yau considered the existence of dHYM metric under assumption of $\mathcal{C}$-subsolutions. They also get the following result for the line bundle MCF.

\begin{thm}[\cite{CJY}, Remark 7.4]
If $\hat \theta>\frac{(n-1)\pi}{2}$ and $u_0$ is a subsolution with $\theta(F\replaced{_{u_0}}{(0)})>\frac{(n-1)\pi}{2}$, then the line bundle MCF starting from $u_0$ converges smoothly to a dHYM metric.
\end{thm}

In \cite{T2019}, Takahashi proved the collapsing behavior of the line bundle \replaced{MCF}{mean curvature flow} on K\"ahler surfaces with boundary conditions on some cohomology class.

Another desirable property of a \replaced{geometric}{geometry} flow is \added{the} stability of stationary points.  That is, if the initial date is sufficiently close to one stationary point, then the flow exists for long time and converges to the stationary point. The stability result gives more evidences that the method of flows will work to find the stationary point. There have been a lot of stability results for other geometric flow\added{s}. In \cite{CL}, Chen and Li gave some stability results of K\"ahler-Ricci flow with respect to the deformation of the underlying complex structures under the assumption $c_1(M)>0$ and no non-zero holomorphic vector fields. In \cite{Z}, Zhu proved that the K\"ahler-Ricci flow converges to a K\"ahler-Einstein metric in $C^\infty$ sense if the initial metric is very close to a K\"ahler-Einstein metric on a Fano manifold. \added{In \cite{LW}, }Lotay\replaced{-}{ and }Wei \deleted{in \cite{LW} }proved that the torsion-free $G_2$ structures are dynamically stable along the Laplacian flow for closed $G_2$ structures. \replaced{Guedj-Lu-Zeriahi}{Guedj, Lu and Zeriahi} considered the stability of solutions to complex Monge-Amp\`ere flows in \cite{GLZ}.

We consider the stability of the line bundle MCF in this paper. That is,  \added{if} we assume \added{that} there exists one dHYM metric on holomorphic line bundle over a compact K\"ahler manifold and the initial metric is $C^2$ close to this dHYM metric, then the flow will admit long-time solution and exponentially converge to this given dHYM metric.

\begin{thm}
\label{thm-stability}
Let $(X,\omega)$ be a \deleted{connected}compact K\"ahler manifold of complex dimension $n$ and $L$ be a holomorphic line bundle over $X$. Assume $\hat h$ is a dHYM metric on $L$ and $h(t)=e^{-u_t}\hat h$ satisfy the line bundle \replaced{MCF}{mean curvature flow}, i.e,
\[
\frac{d}{dt}u\added{_t}=\theta\added{(F_{u_t})}-\hat\theta.
\]
There exists a constant $\delta_0>0$ such that if \added{the smooth initial data $u_0$ satisfies} $||D^2 u_0||_{L^\infty}\leq \delta_0 $, then the line bundle \replaced{MCF}{mean curvature flow} exists for long time and $u_t$ converges to a constant exponentially.
\end{thm}

We do not need any assumption on the phase $\hat \theta$ and the positivity of $F\replaced{_{u_t}}{(t)}$ which are crucial to guarantee the concavity of the operator $\theta(F)$. The method to prove Theorem \ref{thm-stability} is to get uniform estimates of $u_t$. The main step in the proof of this theorem is to obtain that the smallness of $D^2u\added{_t}$ is preserved along the line bundle MCF. As an application of this smallness, we get a uniform estimate for $\nabla \bar \nabla \nabla u_t$ \added{by some parabolic Calabi-type computation}. We should mention that we can not apply the Evans-Krylov theory \added{to get $C^{2,\alpha}$ estimate} here as in \cite{JY}, since we do not have the concavity condition on the operator \added{$\theta$} here. Higher order uniform estimates \added{along the line bundle MCF} are obtained by the standard parabolic Schauder estimate.

We will organize this paper as follows. In Section \ref{section-2}, we review \added{some basics on} the dHYM metrics and \replaced{give}{list} some useful notations and equations. In Section \ref{section-3}, we compute the evolution equations for $u$ and the derivatives of $u$ up to order $2$ along \added {the} line bundle \replaced{MCF}{mean curvature flow}. In Section \ref{section-4}, we \replaced{obtain}{proved} the smallness of $D^2u_t$ \deleted{is preserved}along the line bundle MCF. In Section \ref{section-5}, we prove the first part of Theorem \ref{thm-stability}, i.e. the long-time existence and convergence of \added{the} line bundle \replaced{MCF}{mean curvature flow} \added{subsequently}. In Section \ref{section-expon}, we prove the exponential convergence of \added{the} line bundle MCF which is the second part of Theorem \ref{thm-stability}.

\section{Preliminaries}
\label{section-2}
Let $(X,\omega\deleted{,J})$ be a compact K\"ahler manifold of complex dimension $n$ where\deleted{$J$ is a fixed complex structure and} $\omega$ is a \added{given} K\"ahler form. We denote $(L,h)$ to be a holomorphic line bundle over $X$ and $F$ to be the Chern curvature with respect to the Hermitian metric $h$ on $L$. In local coordinates, we write \[
\omega=\frac{\sqrt{-1}}{2}g_{i\bar j} dz^i\wedge d\bar z^j
\]
and
\[
F=\frac{1}{2} F_{i\bar j} dz^i\wedge d\bar z^j=-\frac{1}{2}\partial_i\partial_{\bar j} \log (h) dz^i\wedge d\bar z^j.
\]
Since $\omega$ is K\"ahler, we have the following second Bianchi equality
\[
\begin{split}
F_{i\bar j,k}&=-\partial_k\partial_{\bar j}\partial_{i}\log(h)+ \Gamma^{s}_{ik} F_{s \bar j}\\
&=-\partial_i\partial_{\bar j}\partial_{k}\log(h)+ \Gamma^{s}_{ki} F_{s \bar j}=F_{k\bar j,i},
\end{split}
\]
i.e. $F_{i\bar j,k}=F_{k\bar j,i}$ where ``$,k$'' is covariant derivative with respect to $\omega$.

Using notations above, we introduce a Hermitian (usually not K\"ahler) metric \added{$\eta$ on $X$} and an endomorphism $K$ of $T^{1,0}(X)$ that are defined by
\begin{equation}
\label{eqn-eta-F}
\eta=\frac{\sqrt{-1}}{2}\eta_{\bar{k}j}dz^j\wedge d\bar z^k
\end{equation}
and
\[
K:=\omega^{-1}F=g^{i\bar j}F_{k\bar j} dz^i\otimes \frac{\partial }{\partial z^k}
\]
where
\[ \eta_{\bar{k}j}=g_{\bar{k}j}+F_{\bar{k}l}g^{l\bar{m}}F_{\bar{m}j}.
\]
Then the complex-valued $(n,n)$-form $(\omega-F)^n$ can be locally written as
\begin{eqnarray*}
(\omega-F)^n&=&n!\det({g_{\bar{k}j}+\sqrt{-1}F_{\bar{k}j}})(\frac{\replaced{\sqrt{-1}}{i}}{2})^n dz^1\wedge d\bar{z}^1\wedge\cdots\wedge dz^n\wedge d\bar{z}^n\\
&=& \det(g^{j\bar{k}})\det({g_{\bar{k}j}+\sqrt{-1}F_{\bar{k}j}})\omega^n\\
&=&\det{(I+\sqrt{-1}K)}\omega^n.
\end{eqnarray*}
Therefore the complex function $\zeta$ can be expressed as
\begin{equation}\label{zeta}
\zeta=\det(I+\sqrt{-1}K).
\end{equation}
Using this expression, we can get the first variation of $\theta(F)$ (cf. Lemma 3.3 in \cite{JY})
\begin{eqnarray*}
\delta\theta(F)=\operatorname{\mathop{Tr}}((I+K^2)^{-1}\delta K)
\end{eqnarray*}
where the endomorphism $I+K^2$ \added{of $T^{1,0}X$} can be expressed locally as
\[I+K^2=g^{p\bar{q}}\eta_{\bar{q}l}\frac{\partial}{\partial z^p}\otimes dz^l.
\]
Thus we obtain the following formula on the derivative of $\theta(F)$
\begin{equation}\label{deltatheta}
\begin{split}
&\partial_i\theta=\operatorname{Tr}((I+ K^2)^{-1}\nabla_i  K)\\
=& \eta^{p\bar{q}}g_{\bar{q}l}\nabla_i(g^{l\bar{m}}  F_{\bar{m}p})\\
=& \eta^{p\bar{q}}\nabla_i F_{\bar{q}p}= \eta^{p\bar q} F_{p\bar q,i}.
\end{split}
\end{equation}

Now suppose $\hat h$ the dHYM metric on $L$ and its Chern curvature is denoted by $\hat F$. Since $\hat h$ satisfies the equation (\ref{eqn-dHYM-angle}), $\theta(\hat F)=\hat\theta$
 is a constant. Taking \deleted{the } derivative\added{s} of $\theta(\hat F)$,  we have the following two equalities
\begin{equation}
\label{eqn-1st-derivative}
\begin{split}
0=\operatorname{Tr}((I+\hat K^2)^{-1}\nabla_i \hat K)=\hat \eta^{p\bar q}\hat F_{p\bar q,i},\\
0=\operatorname{Tr}((I+\hat K^2)^{-1}\nabla\replaced{_{\bar j}}{_i} \hat K)=\hat \eta^{p\bar q}\hat F_{p\bar q,\bar j}
\end{split}
\end{equation}
where $\{\hat \eta^{p\bar q}\}$ is the inverse matrix of $\{\hat \eta_{p\bar{q}}\}$. Taking derivatives again on the both sides of \eqref{eqn-1st-derivative}, we obtain the following two equalities for second order derivatives
\begin{equation}
\label{eqn-2nd-derivative}
\begin{split}
  \hat \eta^{p\bar q}\hat F_{p\bar q, \bar ji }&=-\hat \eta^{p\bar q}_{\mathrel{\phantom{i\bar j}},i} \hat F_{p\bar q,\bar j}\\
  &=\hat \eta^{p\bar t}\hat \eta^{s\bar q} \hat \eta_{s\bar t,i} \hat F_{p\bar q,\bar j}\\
  &=\hat \eta^{p\bar t}\hat \eta^{s\bar q} \hat F_{p\bar q,\bar j} g^{a\bar b}(\hat F_{s\bar b,i}\hat F_{a \bar t}+ \hat F_{s\bar b}\hat F_{a\bar t,i})
\end{split}
\end{equation}
and
\begin{equation}
\label{eqn-2nd-derivative2}
\begin{split}
  \hat \eta^{p\bar q}\hat F_{p\bar q, \bar j\bar i }&=-\hat \eta^{p\bar q}_{\mathrel{\phantom{i\bar j}},\bar i} \hat F_{p\bar q,\bar j}\\
  &=\hat \eta^{p\bar t}\hat \eta^{s\bar q} \hat \eta_{s\bar t,\bar i} \hat F_{p\bar q,\bar j}\\
  &=\hat \eta^{p\bar t}\hat \eta^{s\bar q} \hat F_{p\bar q,\bar j} g^{a\bar b}(\hat F_{s\bar b,\bar i}\hat F_{a \bar t}+ \hat F_{s\bar b}\hat F_{a\bar t,\bar i}).
\end{split}
\end{equation}
These equalities will play significant role in our proof of the main theorem.

A very special case is that we take $[F]=c[\omega]$ for some constant $c$. Then we have a trivial \replaced{dHYM}{deformed Hermitian-Yang-Mills} metric $\hat F=c\omega$. In this case, we have
\[
\hat F_{p\bar q,i}=0, \hat F_{p\bar q,\bar i}=0, \hat F_{p\bar q, i\bar j}=0 \text{ and }\hat F_{p\bar q, \bar i\bar j}=0.
\]
This will make our proof much easier as discussed in Section \ref{section-3}.


\section{Evolution inequalities along the line bundle mean curvature flow}
\label{section-3}

In this section,  we assume that $\hat h$ is \replaced{a given}{the} \replaced{dHYM}{deformed Hermitian-Yang-Mills} metric and consider the line bundle MCF  \eqref{mainequation} with the initial metric $h_0=e^{-u_0}\hat h$. We will give the evolution equations of some quantities related to $\replaced{u_t(x)}{u(x,t)}$ along the line bundle MCF. We also show some inequalities of these quantities. \added{For convenience, we will omit the subscript $t$ of $u_t$ if there is no confusion.}

In the following of this paper, we always use the following second order operator
\[
\Delta_\eta=\eta^{p\bar{q}}\partial_p\partial_{\bar{q}}
\]
where $\eta$ is the Hermitian metric defined in \eqref{eqn-eta-F}. We also remark that $\Delta_\eta$ is the linearization operator of Equation \eqref{eqn-dHYM-angle}.

At first, we compute the evolution of $u^2$ along the line bundle MCF.
\begin{yl}
\label{lemma-evo-0}
The evolution equation of $u^2$ along the line bundle MCF is given by
\begin{equation}
\label{eqn-evo-0}
(\frac{\partial}{\partial t}-\Delta_\eta)u^2=2u(\theta(F\added{_u})-\hat\theta-\Delta_\eta u)-2\eta^{p\bar{q}}u_p u_{\bar{q}}.
\end{equation}
\end{yl}

\begin{proof}
According to the line bundle MCF \eqref{mainequation}, we get that
\[
\frac{\partial}{\partial t} u^2 =2u(\theta(F\added{_u})-\hat\theta)
\]
and
\[
\Delta_\eta u^2=2u\Delta_{\eta}u+2 \eta^{p\bar{q}}u_p u_{\bar{q}}.
\]
Thus we obtain the following equation
\[
(\frac{\partial}{\partial t}-\Delta)u^2=2u(\theta(F\added{_u})-\hat\theta-\Delta_\eta u)-2\eta^{p\bar{q}}u_p u_{\bar{q}}.\qedhere
\]
\end{proof}

We denote $\nabla$ to be the $(1,0)$-part of covariant derivative and $|\cdot|_\omega$ to be the norm with respect to $\omega$.
\begin{yl}
  \label{lemma-evo-1}
  The evolution equation of $|\nabla u|_\omega^2$ along the line bundle MCF is given by
\begin{equation}
\label{eqn-evo-1}
\begin{split}
&(\frac{\partial}{\partial t}-\Delta_\eta)|\nabla u|_\omega^2\\
=-&\eta^{p\bar{q}}g^{i\bar j} (u_{ip}u_{\bar j\bar q}+u_{i\bar {q}}u_{\bar j p})-2\eta^{p\bar{q}}g^{i\bar j}R_{p\bar{l}i\bar{q}}u_lu_{\bar{j}}+2\operatorname{\mathop{Re}}(\eta^{p\bar{q}}g^{i\bar j}\hat F_{p\bar{q},i}u_{\bar{j}}).
\end{split}
\end{equation}

\end{yl}

\begin{proof}
Firstly, we deal with the derivative with respect to $t$,
\[
\begin{split}
&\frac{\partial}{\partial t} |\nabla u|_\omega^2\\
=& \frac{\partial}{\partial t} (g^{i\bar j} u_i u_{\bar j})\\
=& g^{i\bar j}(\frac{\partial u}{\partial t})_i u_{\bar{j}}+ g^{i\bar j}u_i(\frac{\partial u}{\partial t})_{\bar{j}} \\
=&g^{i\bar j}\theta_i u_{\bar{j}}+g^{i\bar j}\theta_{\bar{j}}u_i\\
=&g^{i\bar j}\eta^{p\bar{q}}F_{p\bar{q},i}u_{\bar{j}}+g^{i\bar j}\eta^{p\bar{q}} F_{p\bar{q},\bar{j}}u_{i}
\end{split}
\]
where we have used Equation \eqref{deltatheta} in the last ``=''. Since $F_{p\bar q}=\hat F_{p\bar q}+u_{p\bar q}$, we have
\[
\begin{split}
&\frac{\partial}{\partial t} |\nabla u|_\omega^2\\
=&g^{i\bar j}\eta^{p\bar{q}}u_{p\bar{q}i}u_{\bar{j}}+g^{i\bar j}\eta^{p\bar{q}} u_{p\bar{q}\bar{j}}u_{i}+g^{i\bar j}\eta^{p\bar{q}}\hat F_{p\bar{q},i}u_{\bar{j}}+g^{i\bar j}\eta^{p\bar{q}} \hat F_{p\bar{q},\bar{j}}u_{i}.
\end{split}
\]
Secondly, we deal with the ``$\Delta_\eta$''-part,
\begin{equation}
\label{eqn:3-1}
\begin{split}
&\Delta_{\eta}|\nabla u|_\omega^2\\
=&\eta^{p\bar{q}}(g^{i\bar j}u_i u_{\bar{j}})_{p\bar{q}}\\
=&\eta^{p\bar{q}}g^{i\bar j}(u_{ip}u_{\bar{j}\bar{q}}+u_{i\bar {q}}u_{p\bar{j}}+u_{ip\bar{q}}u_{\bar{j}}+u_iu_{\bar{j} p\bar{q}}).
\end{split}
\end{equation}
Applying the Ricci identity, we \replaced{can change orders of derivatives as follow}{have}
\[
\begin{split}
u_{ip\bar{q}}=&u_{pi\bar{q}}=u_{p\bar{q}i}+u_l R_{p\bar{l}i\bar{q}},\\
u_{\bar{i}p\bar{q}}=&u_{\bar{q}p\bar{i}}=u_{p\bar{q} \bar{i}}.
\end{split}
\]
Inserting these two equalities to Equation \eqref{eqn:3-1}, we obtain
\[
\begin{split}
&\Delta_{\eta}|\nabla u|_\omega^2\\
= \repl{+}&\eta^{p\bar{q}}g^{i\bar j}(u_{ip}u_{\bar{j}\bar{q}}+u_{i\bar {q}}u_{p\bar{j}})\\
+&\eta^{p\bar{q}}g^{i\bar j}u_{\bar{j}}(u_{p\bar{q}i}+u_l R_{p\bar{l}i\bar{q}})+\eta^{p\bar{q}}g^{i\bar j} u_iu_{p\bar{q}\bar{j}}.
\end{split}
\]
Therefore, we have the following evolution equation along the line bundle MCF
\[
\begin{split}
&(\frac{\partial}{\partial t}-\Delta_\eta)|\nabla u|_\omega^2\\
=-&\eta^{p\bar{q}}g^{i\bar j} (u_{ip}u_{\bar j\bar q}+u_{i\bar {q}}u_{\bar j p})-\eta^{p\bar{q}}g^{i\bar j} R_{p\bar{l}i\bar{q}}u_lu_{\bar{j}}+\eta^{p\bar{q}}g^{i\bar j}\hat F_{p\bar{q},i}u_{\bar{j}}+\eta^{p\bar{q}}g^{i\bar j} \hat F_{p\bar{q},\bar{j}}u_{i}.
\end{split}
\]
Then we finish the proof of the lemma.
\end{proof}

Next, let us consider the evolution equations of $|\nabla \bar \nabla u|_\omega^2$ and $|\nabla \nabla u|_\omega^2$ along the line bundle \replaced{MCF}{mean curvature flow}. For convenience, we first introduce some notations before \replaced{detailed computations}{that}
\[
\begin{split}
  \Theta&=|\nabla\bar \nabla u|_\omega^2\\
  \Theta'&=|\nabla \nabla u|_\omega^2\\
\end{split}
\]

\begin{yl}
The evolution equation of $\Theta$ along the line bundle \replaced{MCF}{mean curvature flow} is
\label{lemma-evo-2}
\begin{equation}
  \label{eqn-evo-2}
\begin{split}
&(\frac{\partial }{\partial t}-\Delta_{\eta})\Theta\\
=&-\eta^{k\bar l}g^{i\bar j}g^{p\bar q}(u_{i\bar lp}u_{k\bar j \bar q}+u_{i\bar l\bar q}u_{k\bar j p})-2\operatorname{\mathop{Re}}\left(g^{i\bar j}g^{k\bar l}\eta^{p\bar b}\eta^{a\bar q}\eta_{a\bar b,\bar l}F_{p\bar q ,i}u_{k\bar j}\right)\\
&+2\operatorname{\mathop{Re}}\left(g^{i\bar j}g^{k\bar l} \eta^{p\bar q} u_{k\bar j}(u_{a\bar q}R_{i\bar ap\bar l}-u_{a\bar l}R_{i\bar a p\bar q})\right)+2\operatorname{\mathop{Re}}\left(g^{i\bar j}g^{k\bar l} \eta^{p\bar q} u_{k\bar j}\hat F_{p\bar q{,} i\bar l}\right).
\end{split}
\end{equation}

\end{yl}
\begin{proof}
Firstly, we compute the ``$\frac{\partial}{\partial t}$''-part,
\[
\begin{split}
\frac{\partial}{\partial t}\Theta=&\repl{-}g^{i\bar j}g^{k\bar l} (\frac{\partial u}{\partial t})_{i\bar l}u_{k\bar j}+g^{i\bar j}g^{k\bar l} (\frac{\partial u}{\partial t})_{k\bar j}u_{i\bar l}  \\
=&\repl{-}g^{i\bar j}g^{k\bar l}\theta_{i\bar l}u_{k\bar j}+g^{i\bar j}g^{k\bar l}\theta_{{\bar j k}}u_{i\bar l}\\
=&\repl{-}g^{i\bar j}g^{k\bar l}(\eta^{p\bar q} F_{p\bar q, i})_{\bar l}u_{k\bar j}+g^{i\bar j}g^{k\bar l}(\eta^{p\bar q}F_{p\bar q, {\bar j}})_{{k}}u_{i\bar l}\\
=&-2\operatorname{\mathop{Re}}\left(g^{i\bar j}g^{k\bar l}\eta^{p\bar b}\eta^{a\bar q}\eta_{a\bar b,\bar l} F_{p\bar q,i}u_{k\bar j}\right)+2\operatorname{\mathop{Re}}\left(g^{i\bar j}g^{k\bar l} \eta^{p\bar q} F_{p\bar q,i\bar l} u_{k\bar j}\right).\\
\end{split}
\]
Secondly, we compute the ``$\Delta_\eta$''-part,
\[
\begin{split}
\Delta_{\eta}\Theta=&\repl{+}\eta^{p\bar q}(g^{i\bar j}g^{k\bar l}u_{i\bar l}u_{k\bar j})_{p\bar q}\\
=&\repl{+}g^{i\bar j}g^{k\bar l}\eta^{p\bar q} \left(u_{i\bar l p}u_{k\bar j\bar q}+u_{i\bar l \bar q}u_{k\bar j p}\right)+2\operatorname{\mathop{Re}}\left(g^{i\bar j}g^{k\bar l}\eta^{p\bar q}u_{k\bar j}u_{i\bar l p\bar q}\right).
\end{split}
\]
Adding these equations together, we get the following evolution equation for $\Theta$,
\begin{equation}
\label{eqn:3-2}
\begin{split}
&(\frac{\partial }{\partial t}-\Delta_{\eta})\Theta\\
=&-2\operatorname{\mathop{Re}}\left(g^{i\bar j}g^{k\bar l}\eta^{p\bar b}\eta^{a\bar q}\eta_{a\bar b,\bar l}F_{p\bar q, i}u_{k\bar j}\right)+2\operatorname{\mathop{Re}}\left(g^{i\bar j}g^{k\bar l} \eta^{p\bar q} u_{k\bar j}(F_{p\bar q, i\bar l}-u_{i\bar l p\bar q})\right)\\
&-\eta^{k\bar l}g^{i\bar j}g^{p\bar q}(u_{i\bar lp}u_{k\bar j \bar q}+u_{i\bar l\bar q}u_{k\bar j p})\\
=&-2\operatorname{\mathop{Re}}\left(g^{i\bar j}g^{k\bar l}\eta^{p\bar b}\eta^{a\bar q}\eta_{a\bar b,\bar l}F_{p\bar q, i}u_{k\bar j}\right)+2\operatorname{\mathop{Re}}\left(g^{i\bar j}g^{k\bar l} \eta^{p\bar q} u_{k\bar j}(u_{p\bar qi\bar l}-u_{i\bar l p\bar q})\right)\\
&+2\operatorname{\mathop{Re}}\left(g^{i\bar j}g^{k\bar l} \eta^{p\bar q} u_{k\bar j}\hat F_{p\bar q,i\bar l}\right)-\eta^{k\bar l}g^{i\bar j}g^{p\bar q}(u_{i\bar lp}u_{k\bar j \bar q}+u_{i\bar l\bar q}u_{k\bar j p}).\\
\end{split}
\end{equation}

At last, we deal with the forth order terms appeared above. Indeed, by the Ricci identity, we have the following formula
\[
u_{p\bar q i\bar l}-u_{i\bar lp\bar q}
= u_{a\bar q}R_{i\bar ap\bar l}-u_{a\bar l}R_{i\bar a p\bar q}.
\]
Inserting it into Equation \eqref{eqn:3-2}, we get the following equation
\[
\begin{split}
&(\frac{\partial }{\partial t}-\Delta_{\eta})\Theta\\
=&-\eta^{k\bar l}g^{i\bar j}g^{p\bar q}(u_{i\bar lp}u_{k\bar j \bar q}+u_{i\bar l\bar q}u_{k\bar j p})-2\operatorname{\mathop{Re}}(g^{i\bar j}g^{k\bar l}\eta^{p\bar b}\eta^{a\bar q}\eta_{a\bar b,\bar l}F_{p\bar q ,i}u_{k\bar j})\\
&+2\operatorname{\mathop{Re}}\left(g^{i\bar j}g^{k\bar l} \eta^{p\bar q} u_{k\bar j}(u_{a\bar q}R_{i\bar ap\bar l}-u_{a\bar l}R_{i\bar a p\bar q})\right)+2\operatorname{\mathop{Re}}\left(g^{i\bar j}g^{k\bar l} \eta^{p\bar q} u_{k\bar j}\hat F_{p\bar q{,} i\bar l}\right)
\end{split}
\]
which is the result desired.\qedhere
\end{proof}

By the same argument, we can get the evolution equations of $\Theta'=|\nabla\nabla u|_\omega^2$. For the completeness, we also list the detail here.

\begin{yl}
   The evolution equation of $\Theta'$ along the line bundle MCF is given by
\begin{equation}
\label{eqn-double der}
\begin{split}
&\repl{+}(\frac{\partial }{\partial t}-\Delta_{\eta})\Theta'\\
=&-2\operatorname{\mathop{Re}}(g^{i\bar j}g^{p\bar q}\eta^{k\bar b}\eta^{a\bar l} \eta_{a\bar b,p} F_{k\bar l,i}u_{\bar j\bar q})+2\operatorname{\mathop{Re}}(g^{i\bar j}g^{p\bar q}\eta^{k\bar l} \hat F_{k\bar l,ip}u_{\bar j\bar q})\\
&-2\operatorname{\mathop{Re}}\left(g^{i\bar j}g^{p\bar q} \eta^{k\bar l} u_{\bar j \bar q}(u_{ak}R_{i\bar a p\bar l}+u_{ia} R_{k\bar a p\bar l}+u_{ap}R_{i\bar a k\bar l}+u_{a}R_{i\bar a k\bar l,p})\right)\\
&-\eta^{k\bar l}g^{i\bar j} g^{p\bar q}( u_{i p k}u_{\bar j\bar q\bar l}+u_{i p\bar l}u_{\bar j \bar q k}).
\end{split}
\end{equation}
\end{yl}

\begin{proof}
Taking the derivatives of  $\Theta'$ about $t$, we obtain
\[
\begin{split}
\frac{\partial}{\partial t}\Theta'=&\repl{+}g^{i\bar j}g^{p\bar q}(\dot u_{i p}u_{ \bar j\bar q } +u_{i p}\dot u_{\bar j\bar q })\\
=&\repl{+}g^{i\bar j}g^{p\bar q}(\theta_{i p}u_{\bar j\bar q }+u_{i p}\theta_{\bar j\bar q})\\
=&\repl{+}g^{i\bar j}g^{p\bar q}(\eta^{k\bar l}_{\repl{k\bar l},p } F_{k\bar l,i} +\eta^{k\bar l} F_{k\bar l,ip})u_{\bar j\bar q}+ g^{i\bar j}g^{p\bar q}(\eta^{k\bar l}_{\repl{k\bar l},\bar q } F_{k\bar l,\bar j} +\eta^{k\bar l} F_{k\bar l,\bar j\bar q})u_{ip}\\
=&\repl{+}2\operatorname{\mathop{Re}}\left(g^{i\bar j}g^{p\bar q}(-\eta^{k\bar b}\eta^{a\bar l} \eta_{a\bar b,p} F_{k\bar l,i} +\eta^{k\bar l} F_{k\bar l,ip})u_{\bar j\bar q}\right).\\
\end{split}
\]
We also have the following equation for the ``$\Delta_\eta$''-part,
\[
\begin{split}
\Delta_{\eta}\Theta'&=\eta^{k\bar l}(g^{i\bar j} g^{p\bar q}u_{i p}u_{\bar j\bar q})_{\bar l k}\\
&=\eta^{k\bar l}g^{i\bar j} g^{p\bar q}( u_{i p k}u_{\bar j\bar q\bar l}+u_{i p\bar l}u_{\bar j \bar q k})+2\operatorname{\mathop{Re}}\left(\eta^{k\bar l}g^{i\bar j} g^{p\bar q}u_{\bar j\bar q}u_{ i p \bar l k}\right).
\end{split}
\]
Adding them together, we have the following equation
\begin{align*}
&\repl{+}(\frac{\partial }{\partial t}-\Delta_{\eta})\Theta'\\
=&-\eta^{k\bar l}g^{i\bar j} g^{p\bar q}( u_{i p k}u_{\bar j\bar q\bar l}+u_{i p\bar l}u_{\bar j \bar q k})-2\operatorname{\mathop{Re}}\left(\eta^{k\bar l}g^{i\bar j} g^{p\bar q}u_{\bar j\bar q}u_{ i p \bar l k}\right)\\
&-2\operatorname{\mathop{Re}}\left(g^{i\bar j}g^{p\bar q}\eta^{k\bar b}\eta^{a\bar l} \eta_{a\bar b,p} F_{k\bar l,i}\right) +2\operatorname{\mathop{Re}}\left(g^{i\bar j}g^{p\bar q}\eta^{k\bar l} F_{k\bar l,ip}u_{\bar j\bar q}\right)\\
=&-\eta^{k\bar l}g^{i\bar j} g^{p\bar q}( u_{i p k}u_{\bar j\bar q\bar l}+u_{i p\bar l}u_{\bar j \bar q k})+2\operatorname{\mathop{Re}}\left(\eta^{k\bar l}g^{i\bar j} g^{p\bar q}u_{\bar j\bar q}(u_{k\bar lip}-u_{ i p \bar l k})\right)\\
&-2\operatorname{\mathop{Re}}\left(g^{i\bar j}g^{p\bar q}\eta^{k\bar b}\eta^{a\bar l} \eta_{a\bar b,p} F_{k\bar l,i}\right) +2\operatorname{\mathop{Re}}\left(g^{i\bar j}g^{p\bar q}\eta^{k\bar l} \hat F_{k\bar l,ip}u_{\bar j\bar q}\right).\\
\end{align*}
\replaced{Also by}{By} Ricci identity, \deleted{we know that}
\[
u_{k\bar l i p}-u_{i p \bar l k}
=u_{ak}R_{i\bar a \bar l p}+u_{ia} R_{k\bar a \bar l p}+u_{ap}R_{i\bar a \bar l k}+u_{a}R_{i\bar a \bar l k,p}.
\]
Therefore, we have the following evolution equation
\begin{equation}
\begin{split}
&\repl{+}(\frac{\partial }{\partial t}-\Delta_{\eta})\Theta'\\
=&-2\operatorname{\mathop{Re}}(g^{i\bar j}g^{p\bar q}\eta^{k\bar b}\eta^{a\bar l} \eta_{a\bar b,p} F_{k\bar l,i}u_{\bar j\bar q})+2\operatorname{\mathop{Re}}(g^{i\bar j}g^{p\bar q}\eta^{k\bar l} \hat F_{k\bar l,ip}u_{\bar j\bar q})\\
&-2\operatorname{\mathop{Re}}\left(g^{i\bar j}g^{p\bar q} \eta^{k\bar l} u_{\bar j \bar q}(u_{ak}R_{i\bar a p\bar l}+u_{ia} R_{k\bar a p\bar l}+u_{ap}R_{i\bar a k\bar l}+u_{a}R_{i\bar a k\bar l,p})\right)\\
&-\eta^{k\bar l}g^{i\bar j} g^{p\bar q}( u_{i p k}u_{\bar j\bar q\bar l}+u_{i p\bar l}u_{\bar j \bar q k}).\qedhere
\end{split}
\end{equation}
\end{proof}

As an application of the evolution equations above, we can get the following evolution inequalities (Lemma \ref{lemma-evo-estimate-2} and \ref{lemma-evo-estimate-22}) for $\Theta=|\nabla\bar \nabla u|_\omega^2$ and $\Theta'=|\nabla \nabla u|_\omega^2$ along \added{the} line bundle \replaced{MCF}{mean curvature flow}. Before \replaced{presenting}{proving} these two lemmas, we recall the following arithmetic-geometric mean inequality for the trace of the positive definite Hermitian matrices  $\operatorname{\mathop{Tr}}((A-B)(\overline{A-B})^T)$.

\begin{yl}
\label{lemma-ag-trace}
  If $A$ and $B$ are two Hermitian matrixes, then there holds the following inequality
  \[
  \operatorname{\mathop{Tr}}(A\bar{B}^T+B\bar{A}^T)\leq \operatorname{\mathop{Tr}}(A\bar{A}^T+B\bar{B}^T).
  \]
\end{yl}

\begin{yl}
  \label{lemma-evo-estimate-2}
  There exist two positive constants $C_1, C_2$ depending only on the geometry of $(X,\omega)$ and $\hat F$ such that
  \[
  (\frac{\partial }{\partial t}-\Delta_\eta)\Theta\leq \eta^{a\bar q}u_{a\bar s i}u_{s\bar q \bar i}(-\frac{1}{2}+C_1\Theta) +C_2\Theta(1+\Theta).
  \]
\end{yl}

\begin{zj}
  The key of the proof of this lemma and Theorem \ref{thm-smallhessian} is that $\hat F$ is a solution to dHYM and we can apply Equation \eqref{eqn-1st-derivative} and \eqref{eqn-2nd-derivative}.
\end{zj}

\begin{proof}
For convenience, we will simplify these quantities in the normal coordinates system. In fact, we choose normal coordinates near $p\in X$ such that
\[
g_{i\bar j}(p)=\delta_{i\bar j}, u_{i\bar j}(p)=\sigma_i\delta_{i\bar j} \text{ and }\frac{\partial g_{i\bar j}}{\partial z^k}(p)=0.
\]

We should pay attention that we can diagonalize the Hermitian matrix $\{u_{i\bar j}\}$ only, but not the Hermitian matrix $\eta$\added{ and $F$}. However, in the case $[F]=c[\omega]$, the \replaced{unique}{standard} \replaced{dHYM}{deformed Hermitian-Yang-Mills} metric is
  \[
  \hat F=c\omega.
  \]
In this \added{very special} case, $F$ is diagonal automatically and we can simplify the computation. We leave this very special case to the reader. We deal with the general case in this lemma.

Using the notation that $F_{i\bar j}=\hat F_{i\bar j}+u_{i\bar j}$, we first rewrite the evolution equation of $\Theta=|\nabla\bar \nabla u|^2$ at $p\in X$ as follow,

\begin{equation}
\label{eqn-Theta}
\begin{split}
&(\frac{\partial }{\partial t}-\Delta_{\eta})\Theta\\
=&-\eta^{k\bar l}(u_{i\bar lp}u_{k\bar i \bar p}+u_{i\bar l\bar p}u_{k\bar i p})-2\operatorname{\mathop{Re}}\left(\eta^{p\bar b}\eta^{a\bar q}\eta_{a\bar b,\bar i}F_{p\bar q {,}i}\sigma_i\right)\\
&+2\operatorname{\mathop{Re}}\left(\sigma_i(\sigma_q-\sigma_i) R_{i\bar i p\bar q} \eta^{p\bar q} +\sigma_i \eta^{p\bar q}\hat F_{p\bar q{,} i\bar i}\right)\\
=&-\eta^{k\bar l}(u_{i\bar lp}u_{k\bar i \bar p}+u_{i\bar l\bar p}u_{k\bar i p})-2\operatorname{\mathop{Re}}\left(\eta^{p\bar b}\eta^{a\bar q}\eta_{a\bar b,\bar i}(\hat F_{p\bar q ,i}+u_{p\bar q i})\sigma_i\right)\\
&+2\operatorname{\mathop{Re}}\left(\sigma_i(\sigma_q-\sigma_i) R_{i\bar i p\bar q} \eta^{p\bar q} +\sigma_i \eta^{p\bar q}\hat F_{p\bar q{,} i\bar i}\right)\\
=&-\eta^{k\bar l}(u_{i\bar lp}u_{k\bar i \bar p}+u_{i\bar l\bar p}u_{k\bar i p})-2\operatorname{\mathop{Re}}\left(\eta^{p\bar b}\eta^{a\bar q}\eta_{a\bar b,\bar i}u_{p\bar q i}\sigma_i\right)\\
& +2\operatorname{\mathop{Re}}\left(\sigma_i \eta^{p\bar q}\hat F_{p\bar q{,} i\bar i}-\eta^{p\bar b}\eta^{a\bar q}\eta_{a\bar b,\bar i}\hat F_{p\bar q ,i}\sigma_i\right)+2\operatorname{\mathop{Re}}\left( \sigma_i(\sigma_q-\sigma_i) R_{i\bar i p\bar q} \eta^{p\bar q}\right)\\
=&-\eta^{k\bar l}g^{i\bar j}g^{p\bar q}(u_{i\bar lp}u_{k\bar j \bar q}+u_{i\bar l\bar q}u_{k\bar j p})+A_1+A_2+A_3
\end{split}
\end{equation}
where we set
\[
\begin{split}
A_1=&-2\operatorname{\mathop{Re}}\left(\eta^{p\bar b}\eta^{a\bar q}\eta_{a\bar b,\bar i}u_{p\bar q i}\sigma_i\right),\\
A_2=&\repl{+}2\operatorname{\mathop{Re}}\left(\sigma_i \eta^{p\bar q}\hat F_{p\bar q{,} i\bar i}-\eta^{p\bar b}\eta^{a\bar q}\eta_{a\bar b,\bar i}\hat F_{p\bar q ,i}\sigma_i\right),\\
A_3=&\repl{-}2\operatorname{\mathop{Re}}\left( \sigma_i(\sigma_q-\sigma_i) R_{i\bar i p\bar q} \eta^{p\bar q}\right).\\
\end{split}
\]
Then we will estimate $A_1$, $A_2$ and $A_3$ respectively.

\noindent \textbf{Estimate of $A_1$:}

Since $ \eta_{\bar{k}j}=g_{\bar{k}j}+F_{\bar{k}l}g^{l\bar{m}}F_{\bar{m}j}$, we have
\[
\begin{split}
A_1=&-2Re(\eta^{p\bar b}\eta^{a\bar q}( F_{a\bar t,\bar i}F_{t\bar b}+F_{a\bar t}F_{t\bar b,\bar i})u_{p\bar q i}\sigma_i)\\
=&-2Re(\eta^{p\bar b}\eta^{a\bar q}(\hat F_{a\bar t,\bar i}F_{t\bar b}+F_{a\bar t}\hat F_{t\bar b,\bar i})u_{p\bar q i}\sigma_i)\\
   &-2 Re(\eta^{p\bar b}\eta^{a\bar q}(u_{a\bar t,\bar i}F_{t\bar b}+F_{a\bar t}u_{t\bar b,\bar i})u_{p\bar q i}\sigma_i)\\
=&B_1+B_2
\end{split}
\]
where
\[
  B_1=-2Re(\eta^{p\bar b}\eta^{a\bar q}(\hat F_{a\bar t,\bar i}F_{t\bar b}+F_{a\bar t}\hat F_{t\bar b,\bar i})u_{p\bar q i}\sigma_i)
\]
and
\[
B_2=-2 Re(\eta^{p\bar b}\eta^{a\bar q}(u_{a\bar t,\bar i}F_{t\bar b}+F_{a\bar t}u_{t\bar b,\bar i})u_{p\bar q i}\sigma_i).
\]
Since $\eta$ is a positive definite Hermitian matrix, we can write $\eta$ as square of a positive definite Hermitian matrix, i.e. there exists a positive definite Hermitian matrix $\eta_1$ such that $\eta =\eta_1\cdot \eta_1$. \replaced{And hence, in local coordinates}{, i.e.}
\[
\eta^{p\bar q}=\eta_1^{p\bar m}\eta_1^{m\bar q}.
\]

\noindent \underline{Estimate of $B_1$:}

With notations above, we can rewrite $B_1$ such that
\[
\begin{split}
  B_1=&-2Re(\eta_1^{p\bar m}\eta_1^{m\bar b}\eta_1^{a\bar n}\eta_1^{n\bar q}(\hat F_{a\bar t,\bar i}F_{t\bar b}+F_{a\bar t}\hat F_{t\bar b,\bar i})u_{p\bar q i}\sigma_i).
\end{split}
\]
For convenience, we sperate $B_1$ into two real parts $B_{1,1}$ and $B_{1,2}$ where
\[
B_{1,1}=-2Re(\eta_1^{p\bar m}\eta_1^{m\bar b}\eta_1^{a\bar n}\eta_1^{n\bar q}\hat F_{a\bar t,\bar i}F_{t\bar b}u_{p\bar q i}\sigma_i)
\]
and
\[
B_{1,2}=-2Re(\eta_1^{p\bar m}\eta_1^{m\bar b}\eta_1^{a\bar n}\eta_1^{n\bar q}F_{a\bar t}\hat F_{t\bar b,\bar i}u_{p\bar q i}\sigma_i).
\]
According to Lemma \ref{lemma-ag-trace}, we obtain that
\[
\begin{split}
  B_{1,1}=&-2\operatorname{\mathop{Re}}(\sigma_i \eta_1^{p\bar m} u_{p\bar q i}\eta_1^{n\bar q} \cdot \eta_1^{a\bar n}\hat F_{a\bar t,\bar i}F_{t\bar b}\eta_1^{m\bar b})\\
  \leq & \frac{1}{400}\eta^{p\bar b}\eta^{a\bar q} u_{p\bar q i}u_{\bar b p \bar i} +C\sigma^2_i\eta^{p\bar b}\eta^{a\bar q}F_{p\bar t}\hat F_{t\bar q,i}\hat F_{a\bar s,\bar i}F_{s\bar b}\\
  \leq & \frac{1}{400}\eta^{p\bar b}\eta^{a\bar q} u_{p\bar q i}u_{\bar b p \bar i} +C\Theta F_{s\bar b}\eta^{p\bar b}F_{p\bar t}\hat F_{t\bar q,i}\eta^{a\bar q}\hat F_{a\bar s,\bar i}\\
  =&\frac{1}{400}\eta^{p\bar b}\eta^{a\bar q} u_{p\bar q i}u_{\bar b p \bar i} +C\Theta X_{s\bar t}Y_{t\bar s}
\end{split}
\]
where $X_{s\bar t}=F_{s\bar b}\eta^{p\bar b}F_{p\bar t}$ and $Y_{t\bar s}=\hat F_{t\bar q,i}\eta^{a\bar q}\hat F_{a\bar s,\bar i}$. It is easy to see that $X$ and $Y$ are all semi-positive definite Hermitian matrix. Since $F_{i\bar j}=\hat F_{i\bar j}+u_{i\bar j}$, we have
\begin{equation}
\label{eqn-FhatF}
-C(1+\sqrt{\Theta})I\leq F\leq C(1+\sqrt{\Theta})I.
\end{equation}
 By the definition of $\eta$ \added{and the choice of normal coordinates system}, we know
\[
\eta=I+F^2.
\]
Thus \added{$F$ and $\eta^{-1}$ communicate in the sense of matrixes multiplication, i.e.}
\[
\eta^{-1}F=F\eta^{-1}.
\]
In this sense, we can rewrite $X=\eta^{-1}F^2$ and
\[
\added{0\leq}X\leq C(1+\sqrt{\Theta})^2\eta^{-1}\leq C(1+\Theta)\eta^{-1}\leq C(1+\Theta)I.
\]
Since $Y$ is also semi-positive definite and $\eta>I$, we have
\[
\Theta X_{s\bar t}Y_{t\bar s}\leq C(\Theta+\Theta^2) \hat F_{s\bar q,i}\eta^{a\bar q}\hat F_{a\bar s,\bar i}\leq C(\Theta+\Theta^2).
\]
Thus we have the following inequality for $B_{1,1}$
\begin{equation}
  \label{eqn-B_11estimate}
  B_{1,1}\leq \frac{1}{400} \eta^{p\bar b}\eta^{a\bar q}u_{p\bar q i} u_{\bar bp\bar i}+C(\Theta+\Theta^2).
\end{equation}
The similar inequality also holds for $B_{1,2}$, i.e.
\begin{equation}
  \label{eqn-B_12estimate}
  B_{1,2}\leq \frac{1}{400} \eta^{p\bar b}\eta^{a\bar q}u_{p\bar q i} u_{\bar bp\bar i}+C(\Theta+\Theta^2).
\end{equation}
Adding inequalities \eqref{eqn-B_11estimate} and \eqref{eqn-B_12estimate} together,  we get that
\begin{equation}
  \label{eqn-B_1estimate}
  B_{1}\leq \frac{1}{200} \eta^{p\bar b}\eta^{a\bar q}u_{p\bar q i} u_{\bar bp\bar i}+C(\Theta+\Theta^2).
\end{equation}

\noindent \underline{Estimate of $B_2$:}

The estimate of $B_2$ is similar to that of $B_1$. We \added{also} sperate $B_2$ into two real parts $B_{2,1}$ and $B_{2,2}$ where
\[
B_{2,1}=-2\operatorname{\mathop{Re}}(\sigma_i\eta^{p\bar b}\eta^{a\bar q}u_{a\bar t,\bar i}F_{t\bar b}u_{p\bar q i})
\]
and
\[
B_{2,2}=-2\operatorname{\mathop{Re}}(\sigma_i\eta^{p\bar b}\eta^{a\bar q}F_{a\bar t}u_{t\bar b,\bar i}u_{p\bar q i}).
\]
According to Lemma \ref{lemma-ag-trace}, we have
\begin{equation}
\label{eqn-B_212}
\begin{split}
  B_{2,1}&\leq C(\varepsilon) \sigma_i^2 \eta^{p\bar b}\eta^{a\bar q} u_{p\bar q i}u_{a\bar b \bar i} +\varepsilon \eta^{p\bar b}\eta^{a\bar q} u_{a\bar s \bar i}F_{s\bar b} u_{t\bar q i}F_{p\bar t}\\
  &\leq C(\varepsilon) \Theta \eta^{p\bar b}\eta^{a\bar q} u_{p\bar q i}u_{a\bar b \bar i} +\varepsilon F_{p\bar t}\eta^{p\bar b}F_{s\bar b} u_{a\bar s \bar i}\eta^{a\bar q} u_{t\bar q i}.
\end{split}
\end{equation}
We also denote two semi-positive definite Hermitian matrices $X$ and $Y$ as follow
\[
X_{s\bar t}=F_{p\bar t}\eta^{p\bar b}F_{s\bar b}
\]
and
\[
Y_{t\bar s}=u_{a\bar s \bar i}\eta^{a\bar q} u_{t\bar q i}.
\]
Then the inequality \eqref{eqn-B_212} can be written as
\begin{equation}
\label{eqn-B_211}
B_{2,1}\leq C(\varepsilon) \Theta\eta^{p\bar b}\eta^{a\bar q} u_{p\bar qi}u_{a\bar b\bar i} +\varepsilon \mathrm{Tr}(XY).
\end{equation}
\replaced{As discussed in the estimate of $B_{1,1}$,}{Thus} the Hermitian matrix $X$ satisfies the following inequality
\[
0\leq X\leq C(1+\Theta) \eta^{-1}.
\]
Combining this inequality with \eqref{eqn-B_211} and choosing suitable $\varepsilon$, we obtain
\begin{equation}
\label{eqn-B_21estimate}
\begin{split}
B_{2,1}&\leq C(\varepsilon) \Theta \eta^{p\bar b}\eta^{a\bar q} u_{p\bar qi}u_{a\bar b\bar i} +C\varepsilon(1+\Theta)\eta^{p\bar b} u_{a\bar b\bar i}\eta^{a\bar q} u_{p\bar q i}\\
&\leq (\frac{1}{400}+C\Theta)\eta^{p\bar b}\eta^{a\bar q} u_{p\bar qi}u_{a\bar b\bar i}.
\end{split}
\end{equation}
Similarly, we also have the following inequality for $B_{2,2}$,
\begin{equation}
\label{eqn-B_22estimate}
\begin{split}
B_{2,2}&\leq (\frac{1}{400}+C\Theta)\eta^{p\bar b}\eta^{a\bar q} u_{p\bar qi}u_{a\bar b\bar i}.
\end{split}
\end{equation}
Adding \eqref{eqn-B_21estimate} and \eqref{eqn-B_22estimate} together, we can get the following inequality for $B_2$,
\begin{equation}
\label{eqn-B_2estimate}
\begin{split}
B_{2}&\leq (\frac{1}{200}+C\Theta)\eta^{p\bar b}\eta^{a\bar q} u_{p\bar qi}u_{a\bar b\bar i}.
\end{split}
\end{equation}
Therefore, we have the following estimate for $A_1$ by adding \eqref{eqn-B_1estimate} and \eqref{eqn-B_2estimate},
\begin{equation}
\label{eqn-A_1estimate}
\begin{split}
A_1&\leq (\frac{1}{100}+C\Theta)\eta^{a\bar q} u_{a\bar t \bar k} u_{\bar q t k}+C_1 \Theta(1+\Theta).
\end{split}
\end{equation}

\noindent \textbf{Estimate of $A_2$:}

We rewrite $A_2$ as follow
\[
\begin{split}
A_2=&\repl{+}2\operatorname{\mathop{Re}}(\sigma_i\eta^{p\bar q}\hat F_{p\bar q,i\bar i}-\sigma_i\eta^{p\bar b}\eta^{a\bar q}\eta_{a\bar b,\bar i}\hat F_{p\bar q,i})\\
=&\repl{-}2\operatorname{\mathop{Re}}(\sigma_i(\eta^{p\bar q}-\hat \eta^{p\bar q})\hat F_{p\bar q,i\bar i}-\sigma_i\eta^{p\bar b}\eta^{a\bar q}(\eta_{a\bar b,\bar i}-\hat \eta_{a\bar b,\bar i})\hat F_{p\bar q,i}\\
&-\sigma_i\eta^{p\bar b}(\eta^{a\bar q}-\hat \eta^{a\bar q})\hat \eta_{a\bar b,\bar i} \hat F_{p\bar q,i}-\sigma_i(\eta^{p\bar b}-\hat \eta^{p\bar b})\hat \eta^{a\bar q}\hat \eta_{a\bar b,\bar i} \hat F_{p\bar q,i}\\
&-\sigma_i(\hat \eta^{p\bar b}\hat \eta^{a\bar q}\hat \eta_{a\bar b,\bar i}\hat F_{p\bar q,i} -\hat \eta^{p\bar q}\hat F_{p\bar q,i\bar i}))\\
=&\repl{-}2\operatorname{\mathop{Re}}(\sigma_i(\eta^{p\bar q}-\hat \eta^{p\bar q})\hat F_{p\bar q,i\bar i}-\sigma_i\eta^{p\bar b}\eta^{a\bar q}(\eta_{a\bar b,\bar i}-\hat \eta_{a\bar b,\bar i})\hat F_{p\bar q,i}\\
&-\sigma_i\eta^{p\bar b}(\eta^{a\bar q}-\hat \eta^{a\bar q})\hat \eta_{a\bar b,\bar i} \hat F_{p\bar q,i}-\sigma_i(\eta^{p\bar b}-\hat \eta^{p\bar b})\hat \eta^{a\bar q}\hat \eta_{a\bar b,\bar i} \hat F_{p\bar q,i})\\
\end{split}
\]
where we have used Equation \eqref{eqn-2nd-derivative} in the last equality. For convenience, we sperate $A_2$ into four parts and set
\[
\begin{split}
  D_1=&\repl{-}2\operatorname{\mathop{Re}}(\sigma_i\eta^{p\bar t}(\hat \eta_{s\bar t}- \eta_{s\bar t})\hat \eta^{s\bar q}\hat F_{p\bar q,i\bar i}),\\
  D_2=&-2\operatorname{\mathop{Re}}(\sigma_i \eta^{p\bar b}\eta^{a\bar t}(\hat \eta_{s\bar t}-\eta_{s\bar t})\hat \eta^{s\bar q}\hat \eta_{a\bar b,\bar i} \hat F_{p\bar q,i}),\\
  D_3=&-2\operatorname{\mathop{Re}}(\sigma_i\eta^{p\bar t}(\hat \eta_{s\bar t}-\eta_{s\bar t})\hat\eta^{s\bar b}\hat \eta^{a\bar q}\hat \eta_{a\bar b,\bar i} \hat F_{p\bar q,i}),\\
  D_4=&-2\operatorname{\mathop{Re}}(\sigma_i\eta^{p\bar b}\eta^{a\bar q}(\eta_{a\bar b,\bar i}-\hat \eta_{a\bar b,\bar i})\hat F_{p\bar q,i}).\\
\end{split}
\]
Due to the expression of $\eta$ and $\hat \eta$, we get that
\[
\begin{split}
\hat \eta_{s\bar t}-\eta_{s\bar t}=&\hat F_{s\bar b}\hat F_{b\bar t}-(\hat F_{s\bar b}+u_{s\bar b})(\hat F_{b\bar t}+u_{b\bar t})\\
=&u_{s\bar b}u_{b\bar t}-\hat F_{s\bar b}u_{b\bar t}-u_{s\bar b}\hat F_{b\bar t}\\
=&\delta_{st} \sigma_s\sigma_t -\hat F_{s\bar t}(\sigma_t+\sigma_s).
\end{split}
\]
Hence, $\hat\eta-\eta$ satisfies the following inequality
\begin{equation}
\label{eqn-hateta-eta}
-C(\Theta +\sqrt{\Theta})I\leq \hat \eta-\eta \leq C(\Theta +\sqrt{\Theta})I
\end{equation}
for some constant $C$ depending only on $\hat F$ and $\omega$. With this inequality, we can estimate $D_1$, $D_2$, $D_3$ and $D_4$ respectively.

\noindent\underline{Estimate of $D_1$:}

According to Lemma \ref{lemma-ag-trace}, we can estimate $D_1$ as follow,
\begin{equation}
\label{eqn-D_1estimate}
\begin{split}
D_1=&2\operatorname{\mathop{Re}}( \sigma_i\eta^{p\bar t}(\hat \eta_{s\bar t}- \eta_{s\bar t})\hat \eta^{s\bar q}\hat F_{p\bar q,i\bar i})\\
=&2\operatorname{\mathop{Re}}( \sigma_i\eta_1^{p\bar\alpha}\hat F_{p\bar q, i\bar i}\hat \eta_1^{\beta\bar q}\cdot \eta_1^{\alpha \bar t} (\hat \eta_{s\bar t}-\eta_{s\bar t})\hat \eta_1^{s\bar \beta})\\
\leq &C\sigma^2_i \hat F_{p\bar q,i\bar i}\eta^{p\bar t} \hat F_{a\bar t,i\bar i}\hat \eta^{a\bar q} +C\eta^{p \bar t}(\hat \eta_{s\bar t}-\eta_{s\bar t}) \hat \eta^{s\bar q}(\hat \eta_{p\bar q} -\eta_{p\bar q})\\
\leq &C\Theta +C(\Theta+\sqrt{\Theta})^2\\
\leq & C\Theta+C \Theta^2
\end{split}
\end{equation}
where we have used the inequality \eqref{eqn-hateta-eta} in the second `$\leq$'.

\noindent \underline{Estimate of $D_2$:}

Similar to the estimate of $D_1$, we have the follow estimate for $D_2$,
\begin{equation}
  \label{eqn-D_2estimate}
  \begin{split}
  D_2=&-2\operatorname{\mathop{Re}}(\sigma_i \eta^{p\bar b}\eta^{a\bar t}(\hat \eta_{s\bar t}-\eta_{s\bar t})\hat \eta^{s\bar q}\hat \eta_{a\bar b,\bar i} \hat F_{p\bar q,i})\\
  \leq &\repl{-}C\eta^{a\bar t}(\hat \eta_{s\bar t}-\eta_{s\bar t})(\hat \eta_{a\bar s} -\eta_{a\bar s})\\
  &+C\sigma_i^2 \eta^{a\bar t}\hat \eta_{a\bar b,i}\eta^{p\bar b}\hat F_{p\bar q,i}\hat \eta^{s\bar q} \hat \eta_{c\bar t,\bar i}\eta^{c\bar d}\hat F_{e\bar d,\bar i} \hat \eta^{e\bar s}\\
  \leq & C(\Theta+\Theta^2)
  \end{split}
\end{equation}
where we have used the inequality \eqref{eqn-hateta-eta} and the fact $0< \eta^{-1},\hat \eta^{-1}< I$ in the second ``$\leq$''.

\noindent \underline{Estimate of $D_3$}

Similar to the estimate of $D_1$ and $D_2$, we can estimate $D_3$ as follow
\begin{equation}
\label{eqn-D_3estimate}
\begin{split}
  D_3=&-2\operatorname{\mathop{Re}}(\sigma_i\eta^{p\bar t}(\hat \eta_{s\bar t}-\eta_{s\bar t})\hat\eta^{s\bar b}\hat \eta^{a\bar q}\hat \eta_{a\bar b,\bar i} \hat F_{p\bar q,i})\\
  \leq &\repl{-} C\eta^{p\bar t}(\hat\eta_{s\bar t} -\eta_{s\bar t})(\hat \eta_{p\bar s} -\eta_{p\bar s})\\
  &+ C \sigma^2_i \eta^{p\bar t}\hat F_{p\bar q,i}\hat \eta^{a\bar q}\hat \eta_{a\bar b,\bar i}\hat \eta^{s\bar b}\hat F_{c\bar t,\bar i}\hat \eta^{c\bar d}\hat \eta_{e\bar d,i}\hat\eta^{e\bar s}\\
  \leq& \repl{+}C(\Theta+\Theta^2)
\end{split}
\end{equation}
where we have used the inequality \eqref{eqn-hateta-eta} and the fact $0\leq \eta^{-1},\hat \eta^{-1}\leq I$ in the second ``$\leq$''.

\noindent \underline{Estimate of $D_4$:}

By the definition of $\eta$, we know that
\[
\eta_{a\bar b,\bar i}-\hat \eta_{a\bar b,\bar i}=u_{a\bar s\bar i}\hat F_{s\bar b}+u_{a\bar b\bar i} \sigma_b+\hat F_{a\bar b\bar i}\sigma_b +\sigma_a\hat F_{a\bar b,\bar i}+\hat F_{a\bar s}u_{s\bar b,\bar i}+\sigma_a u_{a\bar b,\bar i}.
\]
As an application of Lemma \ref{lemma-ag-trace}, we have
\begin{align*}
2\mathrm{Re}\left(\sigma_i \eta^{p\bar b}\eta^{a\bar q}\hat F_{p\bar q,i} (u_{a\bar s,\bar i}\hat F_{s\bar b}+ \hat F_{a\bar s}u_{s\bar b,\bar i})\right)&\leq \frac{1}{200}\eta^{a\bar q}u_{a\bar s,\bar i}u_{s\bar q,i}+C\Theta,\\
2\mathrm{Re}\left(\sigma_i(\sigma_a+\sigma_b)\eta^{p\bar b}\eta^{a\bar q}\hat F_{p\bar q,i} u_{a\bar b,\bar i}\right)&\leq \frac{1}{200}\eta^{p\bar b}\eta^{a\bar q}u_{a\bar b,\bar i}u_{p\bar q,i} +C\Theta^2,\\
2\mathrm{Re}\left(\sigma_i(\sigma_a+\sigma_b)\eta^{p\bar b}\eta^{a\bar q}\hat F_{p\bar q,i} \hat F_{a\bar b,\bar i}\right)&\leq C\Theta.
\end{align*}
Adding these inequalities together, we obtain that
\begin{equation}
  \label{eqn-D_4estimate}
  \begin{split}
  D_4&\leq \frac{1}{200} \eta^{a\bar q}u_{a\bar s,\bar i}u_{s\bar q,i}+ \frac{1}{200}\eta^{p\bar b}\eta^{a\bar q}u_{a\bar b,\bar i}u_{p\bar q,i} +C(\Theta+\Theta^2)\\
  &\leq \frac{1}{100}\eta^{a\bar q}u_{a\bar s,\bar i}u_{s\bar q,i}+C(\Theta+\Theta^2)
  \end{split}
\end{equation}
where we have used the fact that $\eta^{-1}< I$ in the last `$\leq$'.

Adding the inequalities \eqref{eqn-D_1estimate}, \eqref{eqn-D_2estimate}, \eqref{eqn-D_3estimate} and \eqref{eqn-D_4estimate} together, we get the following estimate for $A_2$,
\begin{equation}
\label{eqn-A_2estimate}
A_2\leq \frac{1}{100}\eta^{a\bar q}u_{a\bar s i}u_{s\bar q \bar i}+C(\Theta+\Theta^2).
\end{equation}

\noindent \textbf{Estimate of $A_3$:}

At last, we estimate $A_3$. Since $\eta^{-1}\leq I$ and
\[
-C_0\leq R_{i\bar jk\bar l}\leq C_0
\]
for some positive constant $C_0$, we get the following inequality
\begin{equation}
\label{eqn-A_3estimate}
A_3=2\operatorname{\mathop{Re}}\left( \sigma_i(\sigma_q- \sigma_i) R_{i\bar i p\bar q} \eta^{p\bar q}\right)\leq C\Theta.
\end{equation}

Adding the inequalities \eqref{eqn-A_1estimate}, \eqref{eqn-A_2estimate} and \eqref{eqn-A_3estimate} together and inserting into \eqref{eqn-Theta}, we obtain the inequality desired.\qedhere

\end{proof}

We also have the following estimate for $\Theta'=|\nabla\nabla u|_\omega^2$:

\begin{yl}
  \label{lemma-evo-estimate-22}
  There exists constants $C_1, C_2, C_3$ depending only on $(X,\omega)$ and $\hat F$ such that
  \[
  \begin{split}
  (\frac{\partial }{\partial t} -\Delta_\eta)\Theta'&\leq(-\frac{1}{2}+C_1\Theta+C_2\Theta')\eta^{a\bar l} u_{a\bar s\bar p}u_{s\bar lp}+C_3(\Theta+\Theta^2+\Theta'+\Theta'^2+|\nabla u|_\omega^2).
  \end{split}
  \]
\end{yl}

\begin{proof}
The proof is similar to the proof of Lemma 3.6. We just give the key estimate and omit the detail here.

We choose the same normal coordinate systems near $p\in X$ as in Lemma \ref{lemma-evo-estimate-2}. We should pay attention that we can not diagonalize $u_{ij}$ at the same time. We can rewrite the evolution equation of $\Theta'=|\nabla\nabla u|_\omega^2$ as follow
  \[
\begin{split}
&\repl{+}(\frac{\partial }{\partial t}-\Delta_{\eta})\Theta'\\
=&-\eta^{k\bar l}( u_{i p k}u_{\bar i\bar p\bar l}+u_{i p\bar l}u_{\bar i \bar p k})-2\operatorname{\mathop{Re}}(\eta^{k\bar b}\eta^{a\bar l} \eta_{a\bar b,p} F_{k\bar l,i}u_{\bar i\bar p})+2\operatorname{\mathop{Re}}(\eta^{k\bar l} \hat F_{k\bar l,ip}u_{\bar i\bar p})\\
&-2\operatorname{\mathop{Re}}\left(\eta^{k\bar l} u_{\bar i \bar p}(u_{ak}R_{i\bar a p\bar l}+u_{ia} R_{k\bar a p\bar l}+u_{ap}R_{i\bar a k\bar l}+u_{a}R_{i\bar a k\bar l,p})\right)\\
    =&-\eta^{k\bar l}( u_{i p k}u_{\bar i\bar p\bar l}+u_{i p\bar l}u_{\bar i \bar p k})-2\operatorname{\mathop{Re}}(\eta^{k\bar b}\eta^{a\bar l} \eta_{a\bar b,p} u_{k\bar li}u_{\bar i\bar p})\\
    &-2\operatorname{\mathop{Re}}( \eta^{k\bar l} u_{\bar i \bar p}(u_{ak}R_{i\bar a p\bar l}+u_{ia} R_{k\bar a p\bar l}+u_{ap}R_{i\bar a k\bar l}+u_{a}R_{i\bar a k\bar l,p}))\\
    &+2\operatorname{\mathop{Re}}(\eta^{k\bar l} \hat F_{k\bar l,ip}u_{\bar i\bar p} -\eta^{k\bar b}\eta^{a\bar l} \eta_{a\bar b,p} \hat F_{k\bar l,i}u_{\bar i\bar p} )\\
    =&-\eta^{k\bar l}( u_{i p k}u_{\bar i\bar p\bar l}+u_{i p\bar l}u_{\bar i \bar p k})+A'_1+A'_2+A'_3\\
\end{split}
\]
where we set
\[
\begin{split}
  A'_1=&-2\operatorname{\mathop{Re}}(\eta^{k\bar b}\eta^{a\bar l} \eta_{a\bar b,p} u_{k\bar li}u_{\bar i\bar p}),\\
  A'_2=&-2\operatorname{\mathop{Re}}( \eta^{k\bar l} u_{\bar i \bar p}(u_{ak}R_{i\bar a p\bar l}+u_{ia} R_{k\bar a p\bar l}+u_{ap}R_{i\bar a k\bar l}+u_{a}R_{i\bar a k\bar l,p})),\\
      A'_3=&\repl{-}2\operatorname{\mathop{Re}}(\eta^{k\bar l} \hat F_{k\bar l,ip}u_{\bar i\bar p} -\eta^{k\bar b}\eta^{a\bar l} \eta_{a\bar b,p} \hat F_{k\bar l,i}u_{\bar i\bar p} ).\\
\end{split}
\]

By similar arguments\footnote{The difference here is that we can not diagonalize the holomorphic Hessian $\{u_{ij}\}$ and complex $\{u_{i\bar j}\}$ at the same time. But this does not cause trouble in our application of Lemma \ref{lemma-ag-trace}.} of estimating $A_1$, $A_2$ and $A_3$ in Lemma \ref{lemma-evo-2}, we get the following estimates for $A'_1$, $A'_2$ and $A'_3$
\[
\begin{split}
A'_1\leq &(\frac{1}{100}+C\Theta+C\Theta')\eta^{k\bar b}u_{l\bar bp}u_{k\bar l\bar p} +C(\Theta +\Theta^2 +\Theta'^2)\\
A'_2\leq& \frac{1}{100}\eta^{a\bar l} u_{a\bar s\bar p}u_{s\bar lp} +C(\Theta +\Theta^2 +\Theta'^2)\\
A'_3\leq &C(\Theta'+ |\nabla u|^2_\omega).
\end{split}
\]
Adding them together, we finish the proof of the lemma.\qedhere

\end{proof}

\section{Stability of the line bundle MCF}
\label{section-4}
In this section, we will prove that the $C^2$-norm of $u(,t)$ keeps small along the line bundle MCF as long as $C^2$-norm of $u_0$ is small enough. For convenience, we set
\[
\tilde \Theta =|\nabla \bar \nabla u|_\omega^2+|\nabla \nabla u|_\omega^2=\Theta+\Theta'.
\]
And we consider the following auxiliary function
\[
\begin{split}
Q&=|\nabla \bar \nabla u|^2_\omega +|\nabla \nabla u|^2_\omega+K_1 |\nabla u|^2_\omega+ \frac{K_2}{2}(u-u(p,0))^2\\
&=\Theta+\Theta'+K_1 |\nabla u|^2_\omega+ \frac{K_2}{2}(u-u(p, 0))^2
\end{split}
\]
where $K_1$ and $K_2$ are constant to be determined, $p$ is a fixed point on $X$. Since $X$ is compact, we have the following lemma according to differential mean value formula.

\begin{yl}
\label{lemma-smalle}
  There exists a positive constant $C$ depending only on the bounded geometry $(X,\omega)$ such that
  \[
  \begin{cases}
    Q(\cdot,0)\leq C|D^2u_0|^2_{L^\infty}& \text{, at }t=0\\
    Q(\cdot,t)\geq |D^2u_t(\cdot)|^2 &\text{, at any }t\geq 0.
  \end{cases}
  \]
\end{yl}

\begin{zj}
  We know that $D^2 u$ can be controlled by $Q$ at all time and $Q$ can be controlled by $D^2u$ at time $t=0$. So in the proof of the following theorem, we will consider the smallness of $Q$ along \added{the} line bundle \replaced{MCF}{mean curvature flow} instead of $D^2 u$.
\end{zj}

\begin{thm}
  \label{thm-smallhessian}
  There exists a constant $\delta_0>0$ such that if $||D^2 u_0||_{L^\infty}\leq \delta_0$, then $||D^2 u_t||_{L^\infty}\leq C_\omega\delta_0$ along the line bundle MCF \added{for all $t\geq 0$} and $C_\omega$ is a uniform constant depending only on the bounded geometry of $(X,\omega)$.
\end{thm}

\begin{proof}
According to the arguments above, we have the following evolution inequality for $Q$,
\begin{equation}
\label{eqn-Q-1}
\begin{split}
&(\frac{\partial}{\partial t} -\Delta_\eta) Q\\
\leq &\repl{+}\eta^{p\bar q}(u_{p\bar s i}u_{s\bar q \bar i}+u_{p\bar s\bar i}u_{\bar q s i})(-\frac{1}{2}+C\Theta+C\Theta')\\
& +C(\Theta +\Theta'+\Theta^2+\Theta'^2+|\nabla u|^2_\omega)+K_2 (u-u(p,0))(\theta(F_u)-\hat \theta-\Delta_\eta u)\\
&-K_2 \eta^{p\bar q}u_{p}u_{\bar q}-K_1 \eta^{p\bar q}R_{p\bar q j\bar i}u_i u_{\bar j}- K_1\eta^{p\bar q}(u_{ip}u_{\bar i \bar q}+u_{i\bar q}u_{\bar ip})\\
&+2 K_1 \operatorname{\mathop{Re}}(\eta^{p\bar{q}}\hat F_{p\bar{q},i}u_{\bar{i}}).\\
\end{split}
\end{equation}

For convenience, we set
\[
I_1=2 K_1 \operatorname{\mathop{Re}}(\eta^{p\bar{q}}\hat F_{p\bar{q},i}u_{\bar{i}}).
\]
We can estimate $I_1$ in the following way
\[
\begin{split}
I_1=&2K_1\operatorname{\mathop{Re}}((\eta^{p\bar q}-\hat\eta^{p\bar q})\hat F_{p\bar q, i}u_{\bar i})\\
=&2K_1\operatorname{\mathop{Re}}(\eta^{p\bar t}(\hat \eta_{s\bar t}-\eta_{s\bar t})\hat \eta^{s\bar q}\hat F_{p\bar{q},i}u_{\bar{i}})\\
=&2K_1\operatorname{\mathop{Re}}(\eta^{p\bar t}(\hat F_{s\bar a}\hat F_{a\bar t}-F_{s\bar a}F_{a\bar t})\hat \eta^{s\bar q}\hat F_{p\bar{q},i}u_{\bar{i}})\\
=&2K_1\operatorname{\mathop{Re}}(\eta^{p\bar t}(\hat F_{s\bar t}\sigma_t+\hat F_{s\bar t}\sigma_s-\delta_{st}\sigma_s\sigma_t)\hat \eta^{s\bar q}\hat F_{p\bar{q},i}u_{\bar{i}})\\
\leq &K_1\rho_0(\Theta+\Theta^2)+C(\rho_0)K_1|\nabla u|_\omega^2,
\end{split}
\]
where we have used Equation \eqref{eqn-1st-derivative} in the first ``$=$'' and $\rho_0$ is a constant to be decided later.

To estimate the third term in \eqref{eqn-Q-1}, we denote the following endomorphism $K_u$ of $T^{1,0}X$ for every function $u$
\begin{eqnarray*}
K_u&=&g^{i\bar j}(\hat F_{k\bar j}+u_{k\bar j})\frac{\partial}{\partial z^i}\otimes dz^{k}.
\end{eqnarray*}
Indeed, we can view $K_u$ as a matrix-valued function on $X$. The function $\theta(F_{su})$ can be regarded as a function on $s\in \mathbb{R}$ where
\[
F_{su}=\hat F+s\partial \bar \partial u.
\]
Applying differential mean value theorem to $\theta(F_{su})$, we have
\[
\begin{split}
  &\theta(F_u)-\hat \theta -\Delta_\eta u\\
  =&\theta(F_{u})-\theta(\hat F)-\Delta_\eta u\\
  =&\left. \frac{d\theta(F_{su})}{ds}\right|_{s=\xi\in (0,1)}-(I+K_u^2)^{p\bar q} u_{p\bar q}\\
  =&(I+K_{\xi u}^2)^{p\bar q}u_{p\bar q}-(I+K_u^2)^{p\bar q} u_{p\bar q}\\
  =&((I+\hat K^2)^{-1}(\hat K^2-K_{\xi u}^2) (I+K_{\xi u}^2)^{-1})^{p\bar q} u_{p\bar q}\\
  =&-\operatorname{\mathop{Tr}}((I+\hat K^2)^{-1}(\xi \hat KU+\xi U\hat K+\xi^2 U^2)(I+K_{\xi u}^2)^{-1}U)\\
  \leq & C(1+ \Theta) \Theta
\end{split}
\]where $U=\{u_{i\bar j}\}$ is the complex hessian of $u$ and $C$ is a constant depending only on $\hat F$. Thus, the evolution inequality of $Q$ can be rewritten as follow
\begin{equation}
\label{eqn-QQ}
\begin{split}
(\frac{\partial}{\partial t} -\Delta_\eta) Q\leq& \repl{+}\eta^{p\bar q}(u_{p\bar s i}u_{s\bar q \bar i}+u_{p\bar s\bar i}u_{\bar q s i})(-\frac{1}{2}+C\Theta +C\Theta')\\
&-K_1\eta^{p\bar q}u_{i p}u_{\bar q \bar i}+C(\Theta'+\Theta'^2)-K_1\eta^{p\bar q}u_{i\bar q}u_{p\bar i}\\
&+(C+K_1\rho_0)( \Theta+ \Theta^2)+CK_2(u-u(p, 0))( \Theta+\Theta^2)\\
&+(C+C(\rho_0)K_1)|\nabla u|^2_\omega -K_2 \eta^{p\bar q}u_{p}u_{\bar q}-K_1 \eta^{p\bar q}R_{p\bar q j\bar i}u_i u_{\bar j}\\
=& \repl{+}\eta^{p\bar q}(u_{p\bar s i}u_{s\bar q \bar i}+u_{p\bar s\bar i}u_{\bar q s i})(-\frac{1}{2}+C\Theta +C\Theta')\\
&+Q_1+Q_2+Q_3+Q_4\\
\end{split}
\end{equation}
where we denote
\[
\begin{split}
  Q_1=&-K_1\eta^{p\bar q}u_{i p}u_{\bar q\bar i}+C(\Theta'+\Theta'^2),\\
  Q_2=&-K_1\eta^{p\bar q}u_{i\bar q}u_{p\bar i}+(C+K_1\rho_0)( \Theta+ \Theta^2),\\
  Q_3=&\repl{+}CK_2(u-u(p, 0))( \Theta+\Theta^2),\\
  Q_4=&-K_2 \eta^{p\bar q}u_{p}u_{\bar q}+(C+C(\rho_0)K_1)|\nabla u|^2_\omega-K_1 \eta^{p\bar q}R_{p\bar q j\bar i}u_i u_{\bar j}.
\end{split}
\]

\noindent \textbf{Estimate of $Q_1$:}

We first deal with $Q_1$. Indeed, we can estimate $Q_1$ as follow
\[
\begin{split}
  Q_1=&-K_1 \hat \eta^{p\bar q}u_{ip}u_{\bar q \bar i} -K_1(\eta^{p\bar q}-\hat \eta^{p\bar q}) u_{ip}u_{\bar q \bar i} +C(\Theta'+\Theta'^2)\\
  \leq &-\hat C K_1 u_{ip}u_{\bar p\bar i} -K_1 \eta^{p\bar t}(\hat \eta_{s\bar t}-\eta_{s\bar t})\hat \eta^{s\bar q} u_{ip}u_{\bar q\bar i} +C(\Theta' +\Theta'^2)\\
  \leq &(-\hat CK_1 +C)\Theta' +C\Theta'^2 +CK_1(\Theta+\sqrt{\Theta}) \eta^{p\bar s}\hat \eta^{s\bar q} u_{ip}u_{\bar q\bar i}\\
  \leq &(-\hat CK_1 +C_1)\Theta' +C\Theta'^2 +CK_1(\Theta+\sqrt{\Theta}) \Theta'
\end{split}
\]
where we apply \eqref{eqn-hateta-eta} in the second ``$\leq$''. We define $K_{1,1}$ to be the positive constant satisfying
\[
-\hat CK_{1,1}+C_1=-2.
\]
Hence, if we take $K_1\geq K_{1,1}$, then there holds
\[
\begin{split}
Q_1&\leq -\Theta' +C\Theta'^2 +CK_1(\Theta+\sqrt{\Theta}) \Theta'\\
&=\Theta'(-2+CK_1(\Theta+\sqrt{\Theta})+C\Theta').
\end{split}
\]

\noindent \textbf{Estimate of $Q_2$:}

We estimate it as follow
\begin{align*}
&- K_1\eta^{p\bar q}u_{i\bar q}u_{\bar ip}+(C+K_1\rho_0)(\Theta+\Theta^2)\\
= &-K_1 \eta^{p\bar q}u_{l\bar q}u_{p\bar l}+(C+K_1\rho_0)(1+\Theta)u_{p\bar l}u_{l\bar q} \delta_{pq}\\
=&u_{l\bar q}u_{p\bar l}\eta^{p\bar m}((C+K_1\rho_0)(1+\Theta)\eta_{q\bar m}-K_1\delta_{qm})\\
= &u_{l\bar q}u_{p\bar l}\eta^{p\bar m}\{((C+K_1\rho_0)(1+\Theta)-K_1)\delta_{qm} \\
&\repl{u_{l\bar q}u_{p\bar l}\eta^{p\bar m}}+ (C+K_1\rho_0)(1+\Theta)(\hat F_{q\bar n}+u_{q\bar n})(\hat F_{n\bar m}+u_{n\bar m})\}\\
=&u_{l\bar q}u_{p\bar l}\eta^{p\bar m}\{((C+\rho_0K_1-K_1+(C+\rho_0K_1)\Theta))\delta_{qm}+(C+\rho_0K_1)(1+\Theta)\hat F_{q\bar n}\hat F_{n\bar m}\\
&+(C+\rho_0K_1)(1+\Theta)(\hat F_{q\bar n}u_{n\bar m}+\hat F_{n\bar m}u_{q\bar n}+u_{q\bar n}u_{n\bar m})\}\\
\leq &u_{l\bar q}u_{p\bar l}\eta^{p\bar m}\{C(1+\Theta)\delta_{q m} + (\rho_0 K_1-K_1+CK_1\rho_0 \Theta)\delta_{qm} +K_1\rho_0\hat F_{q\bar n}\hat F_{n\bar m}\\
&+(C+\rho_0K_1)(1+\Theta)(\hat F_{q\bar n}u_{n\bar m}+\hat F_{n\bar m}u_{q\bar n}+u_{q\bar n}u_{n\bar m})\}.
\end{align*}

Since $\hat F$ is a bounded $(1,1)$-form, we can choose $\rho_0\in(0,\frac{1}{4})$ small enough such that the positive definite Hermitian matrix $\hat F \bar{ \hat F}^T$ satisfies that
\[
\rho_0\hat F \bar{ \hat F}^T\leq\frac{1}{4}I.
\]
Then we get the following inequality
\[
\begin{split}
&- K_1\eta^{p\bar q}u_{ip}u_{\bar i \bar q}+(C+K_1\rho_0)(\Theta+\Theta^2)\\
\leq &u_{l\bar q}u_{p\bar l}\eta^{p\bar m}\{C(1+\Theta)\delta_{q m} + (-\frac{1}{2}K_1+CK_1\rho_0 \Theta)\delta_{qm} \\
&+(C+\rho_0K_1)(1+\Theta)(\hat F_{q\bar n}u_{n\bar m}+\hat F_{n\bar m}u_{q\bar n}+u_{q\bar n}u_{n\bar m})\}\\
\leq &u_{l\bar q}u_{p\bar l}\eta^{p\bar q}\{(C_2-\frac{1}{2}K_1) +C(1+K_1)(1+\Theta)(\sqrt{\Theta}+\Theta)\}
\end{split}
\]
where we apply the following inequality for the complex Hessian of $u$
\[
-C\sqrt{\Theta}I\leq \{u_{i\bar j}\}\leq C\sqrt{\Theta}I
\]
in the second ``$\leq$''. Similar to the estimate of $Q_1$, we choose a positive constant $K_{1,2}$ such that
\[
C_2-\frac{1}{2}K_{1,2}=-4.
\]
So if $K_1\geq K_{1,2}$, then there holds
\[
\begin{split}
&- K_1\eta^{p\bar q}u_{ip}u_{\bar i \bar q}+(C+K_1\rho_0)(\Theta+\Theta^2)\\
\leq &u_{l\bar q}u_{p\bar l}\eta^{p\bar q}\{-4 +C(1+K_1)(1+\Theta)(\sqrt{\Theta}+\Theta)\}.
\end{split}
\]

Now we take $K_1=\max\{K_{1,1},K_{1,2}\}$. Therefore, we obtain that
\begin{equation}
\label{eqn-Q-11}
  Q_1\leq \Theta'(-2+C(\Theta+\sqrt{\Theta})+C\Theta')\leq \Theta'(-1+CQ)
\end{equation}
and
\begin{equation}
\label{eqn-Q-2}
\begin{split}
  Q_2&\leq u_{l\bar q}u_{p\bar l}\eta^{p\bar q}\{-4 +C(1+\Theta)(\sqrt{\Theta}+\Theta)\}\\
  &\leq u_{l\bar q}u_{p\bar l}\eta^{p\bar q}(-3 +CQ^2).
\end{split}
\end{equation}

\noindent \textbf{Estimate of $Q_4$:}

Before estimate $Q_3$, we first deal with $Q_4$ and choose a certain constant $K_2$. Since we have chosen $K_1$, we can treat it as a constant. We can rewrite $Q_4$ as follow
\[
\begin{split}
Q_4=&-K_2 \eta^{p\bar q}u_{p}u_{\bar q}+C|\nabla u|^2_\omega-C \eta^{p\bar q}R_{p\bar q j\bar i}u_i u_{\bar j}\\
 \leq &-K_2 \eta^{p\bar q}u_{p}u_{\bar q} +C|\nabla u|^2_\omega,
\end{split}
\]
since the curvature $R$ is bounded and $0< \eta^{-1}< I$. Furthermore, we also rewrite $|\nabla u|_\omega^2$ as follow in the normal coordinates
\[
|\nabla u|^2_\omega =u_iu_{\bar i}=\eta^{p\bar m}\eta_{q \bar m} u_pu_{\bar q}.
\]
Combining with the inequality $0\leq\eta\leq C(1+\Theta)I$ and choosing $K_2$ large enough, we obtain the following inequality for $Q_4$,
\begin{equation}
\label{eqn-Q-4}
\begin{split}
Q_4&\leq \eta^{p\bar m}u_{p}u_{\bar q}(-K_2 \delta_{qm} +C\eta_{q\bar m})\\
&\leq \eta^{p\bar q}u_pu_{\bar q}(-K_2+C+C\Theta)\\
&\leq \eta^{p\bar q}u_pu_{\bar q}(-1+CQ).
\end{split}
\end{equation}


\noindent\textbf{ Estimate of $Q_3$:}

At last, we estimate $Q_3$. Since $K_2$ is also a given constant and $I\leq\eta\leq C(1+\Theta)I$, we obtain the following inequality
\begin{equation}
\label{eqn-Q-3}
\begin{split}
Q_3&=CK_2(u-u(p, 0))(\Theta+\Theta^2)\\
&\leq CQ^{\frac{1}{2}}(1+\Theta)u_{i\bar q}u_{p\bar i}\delta_{pq}=CQ^{\frac{1}{2}}(1+\Theta)u_{i\bar q}u_{p\bar i}\eta^{p\bar m}\eta_{q\bar m}\\
&\leq C\eta^{p\bar q}u_{i\bar q}u_{p\bar i}Q^{\frac{1}{2}}(1+\Theta^2)\\
&\leq \eta^{p\bar q}u_{i\bar q}u_{p\bar i}(1+CQ^3).
\end{split}
\end{equation}
Inserting \eqref{eqn-Q-11}, \eqref{eqn-Q-2}, \eqref{eqn-Q-3} and \eqref{eqn-Q-4} into \eqref{eqn-QQ}, we get that
\begin{equation}
\label{eqn-Q-ineq}
\begin{split}
&(\frac{\partial}{\partial t}-\Delta_\eta)Q\\
\leq &\repl{+}\eta^{p\bar q}(u_{p\bar s i}u_{s\bar q \bar i}+u_{p\bar s\bar i}u_{\bar q s i})(-\frac{1}{2}+C\Theta +C\Theta')\\
&+\eta^{p\bar q}u_{l\bar q}u_{p\bar l}(-2+CQ^2+CQ^3) +\Theta'(-1+CQ)+\eta^{p\bar q}u_pu_{\bar q}(-1+CQ)\\
\leq &\repl{+}\eta^{p\bar q}(u_{p\bar s i}u_{s\bar q \bar i}+u_{p\bar s\bar i}u_{\bar q s i})(-\frac{1}{2}+C_1Q)\\
&+\eta^{p\bar q}u_{l\bar q}u_{p\bar l}(-1+C_2Q^3) +(\Theta'+\eta^{p\bar q}u_pu_{\bar q})(-1+C_3Q)\\
\end{split}
\end{equation}
where $C_1, C_2, C_3$ is constant independent of $u$. We choose $ \delta'_0>0$ a small positive constant such that if $0\leq Q\leq \delta'_0$, there holds
\[
\begin{split}
-\frac{1}{2}+C_1Q&< 0;\\
-1+C_2Q^3 &< 0;\\
-1+C_3Q&< 0.
\end{split}
\]
Then we know that $Q\leq \delta'_0$ is preserved along the line bundle MCF by the maximal principle. And hence, there exists a positive constant $\delta_0=\frac{\delta'_0}{C_\omega}$ such that if $|D^2u_0|_{L^\infty}\leq \delta_0$, then $||D^2 u||_{L^\infty}\leq C_\omega\delta_0$ along the line bundle MCF according to Lemma \ref{lemma-smalle}.
\end{proof}

\section{Long time existence of the line bundle MCF}
\label{section-5}

In this section we will prove the first part of Theorem \ref{thm-stability}, i.e. long-time existence of \added{the} line bundle \replaced{MCF}{mean curvature flow}. In order to prove this, we need to get \added{uniform} estimates for higher order derivatives of $u$. According to standard Schauder method, it is enough to obtain the uniform bound of $|\nabla \nabla \bar \nabla u(,t)|^2 $.

\begin{zj}
  Unlike \cite{JY}, our assumptions can not guarantee the positivity of $F_{u_t}$ along \added{the} line bundle \replaced{MCF}{mean curvature flow}. So the operator
  \[
  \theta(F_u)=\sum \arctan \lambda_i(F_u)
  \]
  under our consideration need not be concave and we can not apply Evans-Krylov theory to get $C^{2,\alpha}$ estimate for $u$ directly. To overcome this difficulty, we will use the parabolic Calabi type estimate for \added{the} line bundle \replaced{MCF}{mean curvature flow} which is new to our best knowledge. However, our estimate rely on the smallness of $D^2 u$ along \added{the} line bundle \replaced{MCF}{mean curvature flow}. The usual (parabolic)Calabi type estimate under assumption on the bound of $||u||_{C^2}$ is still an open problem in general case.
\end{zj}

For convenience, we \replaced{denote two notations of the third and forth order derivativs}{set}
\[
\Gamma=|\nabla \nabla \bar \nabla u(,t)|_\omega^2
 \]
 and
 \[
 \Xi=|\nabla \nabla \bar \nabla \nabla u|_\omega^2+|\nabla \bar \nabla \bar \nabla \nabla u|_\omega^2.
 \]
 Before proving Theorem \ref{thm-stability}, we first prove the following lemma.

\begin{yl}
\label{lemma-evo-3}
If $||u||_{C^2}$ is uniformly bounded along \added{the} line bundle MCF, then $\Gamma=|\nabla \bar  \nabla \nabla u|^2_\omega$ satisfies the following inequality
\[
(\frac{\partial}{\partial t}-\Delta_\eta) \Gamma
\leq C+C\Gamma^2-C\Xi
\] where $C$ is a constant dependent only on $||u||_{C^2}$, $\hat F$, $\omega$ and $n$.
\end{yl}

\begin{proof}
We first compute the evolution equation of $\Gamma$ along \added{the} line bundle MCF. The $\frac{\partial}{\partial t}$-part of $\Gamma$ is given by
\begin{align*}
&\frac{\partial }{\partial t}\Gamma\\
=&g^{i\bar a}g^{b\bar j}g^{k\bar c}(\frac{\partial u}{\partial t})_{i\bar j k}u_{\bar a b \bar c}+g^{i\bar a}g^{b\bar j}g^{k\bar c}(\frac{\partial u}{\partial t})_{\bar a b\bar c }u_{ i\bar j  k} \\
=&g^{i\bar a}g^{b\bar j}g^{k\bar c}\theta_{i\bar j k}u_{\bar a b \bar c}+g^{i\bar a}g^{b\bar j}g^{k\bar c}\theta_{\bar a b\bar c }u_{ i\bar j  k} \\
=&g^{i\bar a}g^{b\bar j}g^{k\bar c}(\eta^{p\bar q} F_{p\bar q, i})_{\bar j k} u_{\bar a b \bar c}+g^{i\bar a}g^{b\bar j}g^{k\bar c}(\eta^{p\bar q} F_{p\bar q, \bar a})_{b \bar c} u_{ i \bar j k }\\
=&2\mathop{\operatorname{Re}}\{(\eta^{p\bar q}_{\repl{p\bar q},\bar j k}F_{p\bar q,i}+\eta^{p\bar q}_{\repl{p\bar q},\bar j}F_{p\bar q,ik}+\eta^{p\bar q}_{\repl{p\bar q},k} F_{p\bar q,i\bar j}+\eta^{p\bar q}F_{p\bar q,i\bar jk})u_{\bar i j\bar k}\}\\
=&2\mathop{\operatorname{Re}}(\eta^{p\bar q}(\hat F_{p\bar q,i\bar jk}+u_{p\bar q,i\bar jk})u_{\bar i j\bar k})\\
-& 2\mathop{\operatorname{Re}}(\eta^{p\bar n}\eta^{m\bar q}(F_{m\bar l,\bar j}F_{l\bar n}+F_{m\bar l}F_{l\bar n,\bar j})F_{p\bar q,ik}u_{\bar ij\bar k})\\
-&2\mathop{\operatorname{Re}}(\eta^{p\bar n}\eta^{m\bar q}(F_{m\bar l,k}F_{l\bar n}+F_{m\bar l}F_{l\bar n,k})F_{p\bar q,i\bar j}u_{\bar ij\bar k})\\
+&2\mathop{\operatorname{Re}}(\eta^{p\bar b}\eta^{a\bar n}\eta^{m\bar q}(F_{a\bar c,k}F_{c\bar b}+F_{a\bar c}F_{c\bar b,k})(F_{m\bar l,\bar j}F_{l\bar n}+F_{m\bar l}F_{l\bar n,\bar j})F_{p\bar q,i}u_{\bar ij\bar k})\\
+&2\mathop{\operatorname{Re}}(\eta^{p\bar n}\eta^{m\bar b}\eta^{a\bar q}(F_{a\bar c,k}F_{c\bar b}+F_{a\bar c}F_{c\bar b,k})(F_{m\bar l,\bar j}F_{l\bar n}+F_{m\bar l}F_{l\bar n,\bar j})F_{p\bar q,i}u_{\bar ij\bar k})\\
-&2\mathop{\operatorname{Re}}(\eta^{p\bar n}\eta^{m\bar q}(F_{m\bar l,\bar jk}F_{l\bar n}+F_{m\bar l,\bar j}F_{l\bar n,k}+F_{m\bar l,k}F_{l\bar n, \bar j}+F_{m\bar l}F_{l\bar n,\bar jk})F_{p\bar q,i}u_{\bar ij\bar k}).
\end{align*}

The $\Delta_\eta$-part of $\Gamma$ is given by
\[
\begin{split}
  &\Delta_\eta \Gamma\\
  =&\eta^{p\bar q}(u_{i\bar jk}u_{\bar i j\bar k})_{p\bar q}\\
  =&\eta^{p\bar q}(u_{i\bar jkp}u_{\bar i j\bar k\bar q}+u_{i\bar jkp\bar q}u_{\bar i j\bar k}+u_{i\bar jk\bar q}u_{\bar i j\bar kp}+u_{i\bar jk}u_{\bar i j\bar kp\bar q})\\
  =&2\mathop{\operatorname{Re}}(\eta^{p\bar q}u_{i\bar jkp\bar q}u_{\bar ij\bar k}) +\eta^{p\bar q}(u_{i\bar jkp}u_{\bar i j\bar k\bar q}+u_{i\bar jk\bar q}u_{\bar i j\bar kp}).
\end{split}
\]
Thus, the evolution formula of $\Gamma$ is
\[
\begin{split}
    &(\frac{\partial}{\partial t}-\Delta_\eta) \Gamma\\
  =&2\mathop{\operatorname{Re}}(\eta^{p\bar q}\hat F_{p\bar q,i\bar jk}u_{\bar i j\bar k})+2\mathop{\operatorname{Re}}((u_{p\bar q,i\bar jk}-u_{i\bar jkp\bar q})\eta^{p\bar q}u_{\bar ij\bar k})\\
-& 2\mathop{\operatorname{Re}}(\eta^{p\bar n}\eta^{m\bar q}(F_{m\bar l,\bar j}F_{l\bar n}+F_{m\bar l}F_{l\bar n,\bar j})F_{p\bar q,ik}u_{\bar ij\bar k})\\
-&2\mathop{\operatorname{Re}}(\eta^{p\bar n}\eta^{m\bar q}(F_{m\bar l,k}F_{l\bar n}+F_{m\bar l}F_{l\bar n,k})F_{p\bar q,i\bar j}u_{\bar ij\bar k})\\
+&2\mathop{\operatorname{Re}}(\eta^{p\bar b}\eta^{a\bar n}\eta^{m\bar q}(F_{a\bar c,k}F_{c\bar b}+F_{a\bar c}F_{c\bar b,k})(F_{m\bar l,\bar j}F_{l\bar n}+F_{m\bar l}F_{l\bar n,\bar j})F_{p\bar q,i}u_{\bar ij\bar k})\\
+&2\mathop{\operatorname{Re}}(\eta^{p\bar n}\eta^{m\bar b}\eta^{a\bar q}(F_{a\bar c,k}F_{c\bar b}+F_{a\bar c}F_{c\bar b,k})(F_{m\bar l,\bar j}F_{l\bar n}+F_{m\bar l}F_{l\bar n,\bar j})F_{p\bar q,i}u_{\bar ij\bar k})\\
-&2\mathop{\operatorname{Re}}(\eta^{p\bar n}\eta^{m\bar q}(F_{m\bar l,\bar jk}F_{l\bar n}+F_{m\bar l,\bar j}F_{l\bar n,k}+F_{m\bar l,k}F_{l\bar n, \bar j}+F_{m\bar l}F_{l\bar n,\bar jk})F_{p\bar q,i}u_{\bar ij\bar k})\\
-&\eta^{p\bar q}(u_{i\bar jkp}u_{\bar i j\bar k\bar q}+u_{i\bar jk\bar q}u_{\bar i j\bar kp})\\
\end{split}
\]
By Ricci identity, we have the following formula while changing order of derivatives
\[
\begin{split}
  &u_{p\bar qi\bar jk}-u_{i\bar jkp\bar q}\\
  =&u_{a\bar jk} R_{i\bar a \bar qp}+u_{i\bar a k}R_{a\bar j\bar q p} +u_{i\bar ja}R_{k\bar a\bar q p}+u_{a\bar j p}R_{i\bar a\bar q k}+u_{a\bar j}R_{i\bar a\bar q k,p}\\
  +&u_{a\bar qk}R_{i\bar ap \bar j}++u_{a\bar q}R_{i\bar ap \bar j,k} -u_{i\bar ap} R_{a\bar j\bar qk}-u_{i\bar a k}R_{a\bar q p \bar j}.\\
\end{split}
\]
Therefore, we obtain that
\begin{align*}
  &(\frac{\partial}{\partial t}-\Delta_\eta) \Gamma\\
  =&2\mathop{\operatorname{Re}}(\eta^{p\bar q}\hat F_{p\bar q,i\bar jk}u_{\bar i j\bar k})+2\mathop{\operatorname{Re}}(\eta^{p\bar q}u_{\bar ij\bar k}(u_{a\bar jk} R_{i\bar a \bar qp}+u_{i\bar a k}R_{a\bar j\bar q p} +u_{i\bar ja}R_{k\bar a\bar q p}\\
  +&u_{a\bar j p}R_{i\bar a\bar q k}+u_{a\bar j}R_{i\bar a\bar q k,p}+u_{a\bar qk}R_{i\bar ap \bar j}++u_{a\bar q}R_{i\bar ap \bar j,k} -u_{i\bar ap} R_{a\bar j\bar qk}-u_{i\bar a k}R_{a\bar q p \bar j}))\\
-& 2\mathop{\operatorname{Re}}(\eta^{p\bar n}\eta^{m\bar q}(F_{m\bar l,\bar j}F_{l\bar n}+F_{m\bar l}F_{l\bar n,\bar j})F_{p\bar q,ik}u_{\bar ij\bar k})\\
-&2\mathop{\operatorname{Re}}(\eta^{p\bar n}\eta^{m\bar q}(F_{m\bar l,k}F_{l\bar n}+F_{m\bar l}F_{l\bar n,k})F_{p\bar q,i\bar j}u_{\bar ij\bar k})\\
+&2\mathop{\operatorname{Re}}(\eta^{p\bar b}\eta^{a\bar n}\eta^{m\bar q}(F_{a\bar c,k}F_{c\bar b}+F_{a\bar c}F_{c\bar b,k})(F_{m\bar l,\bar j}F_{l\bar n}+F_{m\bar l}F_{l\bar n,\bar j})F_{p\bar q,i}u_{\bar ij\bar k})\\
+&2\mathop{\operatorname{Re}}(\eta^{p\bar n}\eta^{m\bar b}\eta^{a\bar q}(F_{a\bar c,k}F_{c\bar b}+F_{a\bar c}F_{c\bar b,k})(F_{m\bar l,\bar j}F_{l\bar n}+F_{m\bar l}F_{l\bar n,\bar j})F_{p\bar q,i}u_{\bar ij\bar k})\\
-&2\mathop{\operatorname{Re}}(\eta^{p\bar n}\eta^{m\bar q}(F_{m\bar l,\bar jk}F_{l\bar n}+F_{m\bar l,\bar j}F_{l\bar n,k}+F_{m\bar l,k}F_{l\bar n, \bar j}+F_{m\bar l}F_{l\bar n,\bar jk})F_{p\bar q,i}u_{\bar ij\bar k})\\
-&\eta^{p\bar q}(u_{i\bar jkp}u_{\bar i j\bar k\bar q}+u_{i\bar jk\bar q}u_{\bar i j\bar kp})\\
=&T_1+T_2+T_3+T_4+T_5+T_6+T_7-\eta^{p\bar q}(u_{i\bar jkp}u_{\bar i j\bar k\bar q}+u_{i\bar jk\bar q}u_{\bar i j\bar kp})
\end{align*}
where we set
\[
\begin{split}
T_1=&\repl{-}2\mathop{\operatorname{Re}}(\eta^{p\bar q}\hat F_{p\bar q,i\bar jk}u_{\bar i j\bar k})\\
T_2=&\repl{-}2\mathop{\operatorname{Re}}(\eta^{p\bar q}u_{\bar ij\bar k}(u_{a\bar jk} R_{i\bar a \bar qp}+u_{i\bar a k}R_{a\bar j\bar q p} +u_{i\bar ja}R_{k\bar a\bar q p}+u_{a\bar j p}R_{i\bar a\bar q k}\\
&+u_{a\bar j}R_{i\bar a\bar q k,p}+u_{a\bar qk}R_{i\bar ap \bar j}++u_{a\bar q}R_{i\bar ap \bar j,k} -u_{i\bar ap} R_{a\bar j\bar qk}-u_{i\bar a k}R_{a\bar q p \bar j}))\\
T_3=&- 2\mathop{\operatorname{Re}}(\eta^{p\bar n}\eta^{m\bar q}(F_{m\bar l,\bar j}F_{l\bar n}+F_{m\bar l}F_{l\bar n,\bar j})F_{p\bar q,ik}u_{\bar ij\bar k})\\
T_4=&-2\mathop{\operatorname{Re}}(\eta^{p\bar n}\eta^{m\bar q}(F_{m\bar l,k}F_{l\bar n}+F_{m\bar l}F_{l\bar n,k})F_{p\bar q,i\bar j}u_{\bar ij\bar k})\\
T_5=&\repl{-}2\mathop{\operatorname{Re}}(\eta^{p\bar b}\eta^{a\bar n}\eta^{m\bar q}(F_{a\bar c,k}F_{c\bar b}+F_{a\bar c}F_{c\bar b,k})(F_{m\bar l,\bar j}F_{l\bar n}+F_{m\bar l}F_{l\bar n,\bar j})F_{p\bar q,i}u_{\bar ij\bar k})\\
T_6=&\repl{-}2\mathop{\operatorname{Re}}(\eta^{p\bar n}\eta^{m\bar b}\eta^{a\bar q}(F_{a\bar c,k}F_{c\bar b}+F_{a\bar c}F_{c\bar b,k})(F_{m\bar l,\bar j}F_{l\bar n}+F_{m\bar l}F_{l\bar n,\bar j})F_{p\bar q,i}u_{\bar ij\bar k})\\
T_7=&-2\mathop{\operatorname{Re}}(\eta^{p\bar n}\eta^{m\bar q}(F_{m\bar l,\bar jk}F_{l\bar n}+F_{m\bar l,\bar j}F_{l\bar n,k}+F_{m\bar l,k}F_{l\bar n, \bar j}+F_{m\bar l}F_{l\bar n,\bar jk})F_{p\bar q,i}u_{\bar ij\bar k}).
\end{split}
\]
Since $||u||_{C^2}$ is bounded and $F_{i\bar j}=\hat F_{i\bar j}+u_{i\bar j}$, we can get the following inequalities for $T_i(i=1,\cdots,7)$ according to Cauchy inequality
\[
\begin{split}
T_1+T_2&\leq C+C\Gamma^2,\\
T_3+T_4+T_7&\leq C+C\Gamma^2 +\frac{1}{100}\Xi,\\
T_5+T_6&\leq C+C\Gamma^2,\\
\end{split}
\]
Hence, we get that
\[
(\frac{\partial}{\partial t}-\Delta_\eta) \Gamma \leq  C+C\Gamma^2-C\Xi.\qedhere
\]
\end{proof}

Now we begin to prove the first part of Theorem \ref{thm-stability}. The exponential convergence of will be presented in Section \ref{section-expon}.

\begin{proof}
We assume the maximal existence interval of line bundle MCF is $[0,T)$. Since $\Theta=|\nabla \bar \nabla u|_\omega^2$ is small along the line bundle \replaced{MCF}{mean curvature flow}, we have the following estimate for $\Theta$ by choosing $\delta_0$ small enough
\[
(\frac{\partial}{\partial t}-\Delta_\eta)\Theta \leq -C_1\Gamma+C,
\]
for some positive constants $C_1$ and $C$ according to Lemma \ref{lemma-evo-estimate-2}.

Suppose $A$ is a constant to be determined later. Then we have the following inequality along \added{the} line bundle MCF
\[
\begin{split}
&(\frac{\partial}{\partial t}-\Delta_\eta)(e^{A\Theta}\Gamma)\\
=&e^{A\Theta}(\frac{\partial}{\partial t}-\Delta_\eta)\Gamma +Ae^{A\Theta}\Gamma (\frac{\partial}{\partial t}-\Delta_\eta)\Theta\\
&-2\mathop{\operatorname{Re}}(Ae^{A\Theta}\Theta_p\Gamma_{ \bar q}\eta^{p\bar q})-A^2 e^{A\Theta}\eta^{p\bar q}\Theta_p\Theta_{\bar q}\Gamma\\
\leq & e^{A\Theta}(C+C\Gamma^2 -C\Xi) +A e^{A\Theta}\Gamma(-C_1 \Gamma+C)-A^2e^{A\Theta} \Gamma\eta^{p\bar q}\Theta_p\Theta_{\bar q}\\
&-2\mathop{\operatorname{Re}}(Ae^{A\Theta}\Theta_p\Gamma_{ \bar q}\eta^{p\bar q}).
\end{split}
\]


According to the equation
\[
\nabla (e^{A\Theta}\Gamma) =A e^{A\Theta}\nabla \Theta \Gamma+e^{A\Theta} \nabla\Gamma,
\]
we obtain that
\[
-2\mathop{\operatorname{Re}} (Ae^{A\Theta}\Theta_p\Gamma_{\bar q}\eta^{p\bar q})=-2\operatorname{\mathop{Re}}(A\eta^{p\bar q}\Theta_p(e^{A\Theta}\Gamma)_{\bar q})+2\mathop{\operatorname{Re}}( A^2e^{A\Theta}\Theta_p\Theta_{\bar q}\eta^{p\bar q}\Gamma).
\]
Inserting it into the evolution inequality above, 
\[
\begin{split}
&(\frac{\partial}{\partial t}-\Delta_\eta)(e^{A\Theta}\Gamma)\\
\leq & e^{A\Theta}(C+C\Gamma^2 -C\Xi) +A e^{A\Theta}\Gamma(-C_1 \Gamma+C)-A^2e^{A\Theta} \Gamma\eta^{p\bar q}\Theta_p\Theta_{\bar q}\\
&-2\mathop{\operatorname{Re}}(A\eta^{p\bar q}\Theta_p(e^{A\Theta}\Gamma)_{\bar q})+2A^2e^{A\Theta}\Theta_p\Theta_{\bar q}\eta^{p\bar q}\Gamma\\
=& e^{A\Theta}(C+C\Gamma^2 -C\Xi) +A e^{A\Theta}\Gamma(-C_1 \Gamma+C)+A^2e^{A\Theta} \Gamma\eta^{p\bar q}\Theta_p\Theta_{\bar q}\\
&-2\mathop{\operatorname{Re}}(A\eta^{p\bar q}\Theta_p(e^{A\Theta}\Gamma)_{\bar q}).
\end{split}
\]

Since $|\nabla\bar \nabla u|_\omega^2 \leq \delta_0$ and $\eta^{-1}\leq I$, we get that
\[
\Theta_p\Theta_{\bar q}\eta^{ p\bar q}=(u_{i\bar jp}u_{j\bar i}+u_{i\bar j}u_{j\bar ip} )(u_{k\bar l\bar q}u_{l\bar k}+ u_{k\bar l}u_{l\bar k\bar q})\eta^{p\bar q}\leq \delta_0 \Gamma,
\]
and
\[
2\mathop{\operatorname{Re}}( A^2e^{A\Theta}\Theta_p\Theta_{\bar q}\eta^{p\bar q}\Gamma)\leq 2\delta_0 A^2 \Gamma^2e^{A\Theta}
\]
According to the inequalities above, we get that
\[
\begin{split}
(\frac{\partial}{\partial t}-\Delta_\eta)(e^{A\Theta}\Gamma)\leq& -2\operatorname{\mathop{Re}}(A\eta^{p\bar q}\Theta_p(e^{A\Theta}\Gamma)_{\bar q})\\
&+e^{A\Theta}\{(C_2-AC_1+A^2 \delta_0)\Gamma^2 +AC\times\Gamma+C\}.
\end{split}
\]
We choose $\delta_0$ small enough\footnote{The constant $\delta_0$ chosen here may be smaller than that in Lemma \ref{thm-smallhessian}, so the smallness of Hessian is still preserved along line bundle \replaced{MCF}{mean curvature flow}.} such that
\[
C_1^2-4C_2\delta_0>0.
\]
Then we can choose $A$ such that
\[
-C_3=C_2-AC_1+A^2\delta_0< 0.
\]
Hence $e^{A\Theta}\Gamma$ is bounded along the line bundle \replaced{MCF}{mean curvature flow} by maximal principle. As a consequence, $\Gamma$ is bounded since $\Theta$ is bounded, i.e. $\nabla \bar \nabla \nabla u(,t)$ is uniformly bounded for all $t\in[0,T)$.

Then we get the uniform estimate of higher order derivatives as follow. The uniform bound of $\nabla \bar \nabla \nabla u(,t)$ implies that the $C^{\alpha}$-norm of $\nabla\bar \nabla u(,t)$ is uniformly bounded for any $\alpha\in(0,1)$ and $t\in[0,T)$. Hence the $C^{\alpha}$-norm of $\eta$ is uniformly bounded. The standard parabolic Schauder estimate gives us the uniform higher order estimate. Then we can extend the line bundle \replaced{MCF}{mean curvature flow} across time $T$. As a consequence, we get the long-time existence and convergence of the line bundle \replaced{MCF}{mean curvature flow} in the sense of subsequence.
\end{proof}


\section{Exponential Convergence}
\label{section-expon}
In this section, we prove the exponential convergence as stated in Theorem \ref{thm-stability}. By the line bundle \replaced{MCF}{mean curvature flow} and  $(\ref{deltatheta})$, we get that the following equation for function $\theta$,
\begin{equation*}
\frac{\partial}{\partial t}\theta =\mbox{Tr}((I+K^2)^{-1}\frac{\partial}{\partial t} K)=\eta^{i\bar j}g_{\bar j l}\frac{\partial}{\partial t}(g^{l\bar m}F_{\bar m i})=\eta^{i\bar j} (\dot u)_{i\bar j}=\eta^{i\bar j} \theta_{i\bar j}.
\end{equation*}
Hence according the maximal principle, we know that the maximum and minimum of $\theta(\cdot,t)$ attains at $t=0$, i.e. $\theta$ is bounded.

\subsection{Harnack-type Inequality}

Before considering the exponential convergence of $u(\cdot,t)$, we will first prove a Harnack-type inequality for positive solutions $\varphi$ of the following parabolic equation
 \begin{equation}\label{eqn-dot u}
 \frac{\partial v}{\partial t}=\eta^{i\bar j}v_{i\bar j},
 \end{equation}
where $\eta^{i\bar j}$ is the Hermitian matrix appeared above dependent on $u(x,t)$. This Harnack type inequality has been proved by Li-Yau for heat equation in \cite{LY}. And Cao proved it for K\"ahler-Ricci flow in \cite{Cao}. \added{The argument is standard and we give the details for completeness of this paper.}

For convenience, we set $f=\log v$ and
\[
\tilde f=t(\eta^{i\bar j}f_if_{\bar j} -\alpha \dot f)
\]
where $\alpha $ is a constant in $(1,2)$. By Equation \eqref{eqn-dot u}, we get that
\begin{equation}
\label{eqn-dot f}
\dot f -\eta^{i\bar j}f_{i\bar j}=\eta^{i\bar j}f_if_{\bar j}
\end{equation}
and
\begin{equation}
\label{eqn-F-2}
\tilde f=-t\eta^{i\bar j} f_{i\bar j} -t (\alpha-1)\dot f.
\end{equation}

\begin{yl}
\label{lemma-5.1}
  There exist constants $C_1$ and $C_2$ which depend on the bound of $F$ and the derivatives of $F$ such that the function $\tilde f$ satisfies the following inequality
  \[
  \eta^{k\bar l} \tilde f_{k\bar l} -\dot {\tilde f}\geq \frac{t}{2n}(\eta^{i\bar j}f_if_{\bar j} -\dot f)^2 -2\operatorname{\mathop{Re}}(\eta^{i\bar j} f_i{\tilde f}_{\bar j}) -(\eta^{i\bar j}f_if_{\bar j} -\alpha \dot f) -C_1 t \eta^{i\bar j} f_if_{\bar j} -C_2 t.
  \]
\end{yl}

\begin{proof}
  By direct computation, we have
  \begin{equation}
    \label{eqn-dot F}
    \begin{split}
      \dot {\tilde f}&=\eta^{i\bar j} f_if_{\bar j} -\alpha \dot f+2 t\operatorname{\mathop{Re}}(\eta^{i\bar j}f_{\bar j} \dot f_i) +t \frac{\partial \eta^{i\bar j}}{\partial t} f_if_{\bar j} -\alpha t\ddot f
    \end{split}
  \end{equation}
  and
  \begin{equation}
    \label{eqn-laplace F}
    \begin{split}
      \eta^{k\bar l}{\tilde f}_{k\bar l}=\repl{+}& t\eta^{k\bar l}(\underline{\eta^{i\bar j} f_{ik}f_{\bar j\bar l}}_{F_1} +\underline{\eta^{i\bar j} f_{i\bar l}f_{\bar jk}}_{F_2} +\underline{\eta^{i\bar j}_{\repl{i\bar j},k}f_{i\bar l}f_{\bar j}}_{F_3} +\underline{\eta^{i\bar j}_{\repl{i\bar j},k}f_{i}f_{\bar j \bar l}}_{F_4}\\
      +&\underline{\eta^{i\bar j}f_{ik\bar l}f_{\bar j}}_{F_5} +\underline{\eta^{i\bar j} f_i f_{\bar j k\bar l}}_{F_6}+\underline{\eta^{i\bar j}_{\repl{i\bar j},\bar l}f_{ik} f_{\bar j}}_{F_7} +\underline{\eta^{i\bar j}_{\repl{i\bar j},\bar l}f_{i}f_{\bar jk}}_{F_8} +\underline{\eta^{i\bar j}_{\repl{i\bar j},k\bar l}f_if_{\bar j}}_{F_9}\\
      -&\underline{\alpha \dot f_{k\bar l}}_{F_{10}})\\
      =\repl{+}&\sum_{i=1}^{10} F_i.
    \end{split}
  \end{equation}

  For any $\varepsilon>0$, we have the following inequality
  \[
  \begin{split}
  |F_3+F_8|&\leq \frac{2 t}{\varepsilon}\eta^{i\bar j}f_if_{\bar j}+2\varepsilon F_2,\\
  |F_4+F_7|&\leq \frac{2 t}{\varepsilon}\eta^{i\bar j}f_if_{\bar j}+2\varepsilon F_1
  \end{split}
  \]
  according to Cauchy inequality. Since $\eta^{i\bar j}_{\repl{i\bar j},k\bar l}$ is bounded along \added{the} line bundle MCF, we know that $F_9$ satisfies
  \[
  |F_9|\leq Ct\eta^{i\bar j}f_if_{\bar j}.
  \]
  Furthermore, we can also estimate $F_5+F_6$ and $F_{10}$ as follow
  \begin{align*}
    F_5+F_6&=t\eta^{i\bar j}\eta^{k\bar l}(f_{k\bar li}f_{\bar j} +f_a R_{k\bar ai\bar l}f_{\bar j} +f_{i}f_{k\bar l\bar j})\\
    &\geq -Ct\eta^{i\bar j}f_if_{\bar j} +2 t\operatorname{\mathop{Re}}(\eta^{i\bar j}f_{\bar j}(\eta^{k\bar l}f_{k\bar l})_{i}) -t\eta^{i\bar j}(\eta^{k\bar l}_{\repl{k\bar l},i}f_{k\bar l}f_{\bar j} +\eta^{k\bar l}_{\repl{k\bar l},\bar j}f_{k\bar l}f_{i})\\
    &\geq -Ct\eta^{i\bar j}f_if_{\bar j} +2 t\operatorname{\mathop{Re}}(\eta^{i\bar j}f_{\bar j}(\eta^{k\bar l}f_{k\bar l})_{i}) -\frac{t}{\varepsilon}\eta^{i\bar j}f_if_{\bar j} -\varepsilon tF_2\\
    &=-Ct\eta^{i\bar j}f_if_{\bar j} -2 \operatorname{\mathop{Re}}(\eta^{i\bar j}f_{\bar j}{\tilde f}_{i}) -2t(\alpha-1)\operatorname{\mathop{Re}}(\eta^{i\bar j}f_{\bar j}\dot f_i) -\frac{t}{\varepsilon}\eta^{i\bar j}f_if_{\bar j} -\varepsilon t F_2\\
    &=-Ct\eta^{i\bar j}f_if_{\bar j} -2 \operatorname{\mathop{Re}}(\eta^{i\bar j}f_{\bar j}{\tilde f}_{i})  -(\alpha-1)\dot {\tilde f} +(\alpha-1) (\eta^{i\bar j}f_if_{\bar j} -\alpha \dot f)\\
    &\repl{=}+(\alpha-1) t \frac{\partial \eta^{i\bar j}}{\partial t} f_i f_{\bar j} -\alpha(\alpha-1)t\ddot f -\frac{t}{\varepsilon}\eta^{i\bar j}f_if_{\bar j} -\varepsilon tF_2\\
    &\geq -Ct\eta^{i\bar j}f_if_{\bar j} -2 \operatorname{\mathop{Re}}(\eta^{i\bar j}f_{\bar j}{\tilde f}_{i}) -(\alpha-1)\dot {\tilde f} +(\alpha-1) (\eta^{i\bar j}f_if_{\bar j} -\alpha \dot f)\\
    &\repl{\geq }-\alpha(\alpha-1)t\ddot f -\frac{t}{\varepsilon}\eta^{i\bar j}f_if_{\bar j} -\varepsilon tF_2
  \end{align*}
  and
  \begin{align*}
    F_{10}&=-\alpha t(\frac{\tilde f}{t^2}-\frac{\dot {\tilde f}}{t}-(\alpha-1)\ddot {\tilde f})+\alpha tf_{k\bar l}\frac{\partial\eta^{k\bar l}}{\partial t}\\
    &\geq -\frac{Ct}{\varepsilon}-\varepsilon F_2 -\frac{\alpha {\tilde f}}{t}+\alpha \dot {\tilde f} +t\alpha(\alpha-1)\ddot f.
  \end{align*}
  Adding all inequalities above, we get that
  \begin{equation}
    \label{eqn-laplace F-2}
    \begin{split}
      \eta^{i\bar j} {\tilde f}_{i\bar j}\geq&\repl{+} \dot {\tilde f}-2\operatorname{\mathop{Re}}(\eta^{i\bar j}{\tilde f}_i f_{\bar j}) -(\eta^{i\bar j}f_i f_{\bar j}- \alpha \dot f) +t(1-4\varepsilon) \eta^{i\bar j}\eta^{k\bar l} f_{i\bar l}f_{k\bar j}\\
       &+t(1-2\varepsilon) \eta^{i\bar j}\eta^{k\bar l} f_{ik} f_{\bar j\bar l} -t(C+\frac{5}{\varepsilon})\eta^{i\bar j} f_i f_{\bar j} -\frac{Ct}{\varepsilon}.
    \end{split}
  \end{equation}

  Taking the constant $\varepsilon$ small enough and applying the following arithmetic-geometric mean inequality
  \[
  \eta^{i\bar j} \eta^{k\bar l} f_{i\bar l}f_{k\bar j}\geq \frac{1}{n} \left(\eta^{i\bar j}f_{i\bar j}\right)^2 =\frac{1}{n}(\eta^{i\bar j} f_i f_{\bar j}- \dot f)^2,
  \]
  we obtain that
  \[
    \label{eqn-laplace F-22}
    \begin{split}
      \eta^{i\bar j} {\tilde f}_{i\bar j}-\dot {\tilde f}\geq&\repl{+} \frac{t}{2n} (\eta^{i\bar j} f_i f_{\bar j} -\dot f)^2 -2\operatorname{\mathop{Re}}(\eta^{i\bar j}{\tilde f}_i f_{\bar j}) -(\eta^{i\bar j}f_i f_{\bar j}- \alpha \dot f) \\
       &-C_1t\eta^{i\bar j} f_i f_{\bar j} -C_2t
    \end{split}
  \]
  which is the result desired.
\end{proof}

\begin{yl}
\label{lemma-5.2}
  There exists constants $C_1$ and $C_2$  which depend on $F$ and the derivatives of $F$ such that for all $t>0$, the following inequality holds
  \[
  \eta^{i\bar j}f_if_{\bar j} -\alpha \dot f\leq C_1 +\frac{C_2}{t}.
  \]
\end{yl}

\begin{proof}
  For any fixed $T_0>0$, we assume that ${\tilde f}$ attained its maximum in $X\times [0,T_0]$ at $(x_0,t_0)$. If $t_0=0$, then we get the required inequality. So we can just consider the case that $t_0>0$. Then at $(x_0,t_0)$, by Lemma \ref{lemma-5.1},
  \begin{equation}
  \label{eqn-Harnack-1}
    \frac{t_0}{2n}(\eta^{i\bar j}f_i f_{\bar j} -\dot f)^2-(\eta^{i\bar j}f_i f_{\bar j}-\alpha \dot f) \leq C_1 t_0 \eta^{i\bar j} f_i f_{\bar j} +C_2 t_0.
  \end{equation}
  In the case $\dot f(x_0,t_0)>0$, we have the following inequality
  \[
    \frac{t_0}{2n}(\eta^{i\bar j}f_i f_{\bar j} -\dot f)^2-(\eta^{i\bar j}f_i f_{\bar j}-\dot f) \leq C_1 t_0 \eta^{i\bar j} f_i f_{\bar j} +C_2 t_0
  \]
  since $\alpha\in (1,2)$. Hence, at $(x_0,t_0)$,
  \[
  \eta^{i\bar j} f_i f_{\bar j} -\dot f\leq C_3 \sqrt{\eta^{i\bar j} f_i f_{\bar j}} +C_4 +\frac{C_5}{t_0} \leq \left(1-\frac{1}{\alpha}\right) \eta^{i\bar j} f_i f_{\bar j} +C_6+\frac{C_5}{t_0}.
  \]
  According to $\alpha \in (1,2)$ and $\dot f>0$, there holds
  \[
  \eta^{i\bar j}f_i f_{\bar j}-\alpha \dot f\leq C_6+\frac{C_5}{t_0}.
  \]
  By the definition of ${\tilde f}$,
  \[
  {\tilde f}(x_0,t_0)= t_0(\eta^{i\bar j}f_i f_{\bar j}-\alpha \dot f)\leq C_6 t_0 +C_5.
  \]
  Therefore, for all $x\in M$,
  \[
  {\tilde f}(x,T_0)\leq {\tilde f}(x_0,t_0)\leq C_6 t_0 +C_5\leq C_6 T_0 +C_5
  \]
  i.e.
  \[
  (\eta^{i\bar j}f_i f_{\bar j} -\alpha \dot f)(x,T_0) \leq C+\frac{C}{T_0}.
  \]
  So we complete the proof in this case.

  Now let us consider the case when $\dot f(x_0,t_0)\leq 0$. By the inequality \eqref{eqn-Harnack-1},
  \[
  \frac{t_0}{2n}(\eta^{i\bar j} f_i f_{\bar j})^2 -\eta^{i\bar j} f_i f_{\bar j} +\alpha \dot f \leq C_1 t_0 \eta^{i\bar j} f_i f_{\bar j} +C_2 t_0
  \]
  i.e.
  \[
  \frac{1}{2n}(\eta^{i\bar j}f_i f_{\bar j})^2 -(\frac{1}{t_0} +C_1)\eta^{i\bar j}f_i f_{\bar j}\leq C_2 -\frac{\alpha \dot f}{t_0}.
  \]
  Then by Cauchy inequality,
  \[
  \frac{1}{2n}(\eta^{i\bar j}f_i f_{\bar j})^2 -(\frac{1}{t_0} +C_1)\eta^{i\bar j}f_i f_{\bar j}\leq C_2 + \left(\frac{C}{t_0}\right)^2 + \frac{\dot f^2}{4}
  \]
  Hence, at $(x_0,t_0)$,
  \begin{equation}
  \label{eqn-Harnack-2}
  \eta^{i\bar j} f_i f_{\bar j} \leq C+\frac{C}{t_0}-\frac{\dot f}{2}.
  \end{equation}
  On the other hand, by inequality \eqref{eqn-Harnack-1},
  \[
  \frac{t_0}{n}\dot f^2 +\alpha \dot f\leq C_1 t_0 \eta^{i\bar j} f_i f_{\bar j} +C_2 t_0 +\eta^{i\bar j} f_i f_{\bar j}
  \]
  i.e.
  \[
  \frac{1}{n} \dot f^2 +\frac{\alpha \dot f}{t_0}\leq C_1 \eta^{i\bar j} f_i f_{\bar j} +C_2 +\frac{\eta^{i\bar j} f_i f_{\bar j}}{t_0}.
  \]
  By Cauchy inequality,
  \[
  \frac{1}{n} \dot f^2 +\frac{\alpha \dot f}{t_0}\leq \frac{1}{4} \left( \eta^{i\bar j} f_i f_{\bar j}\right)^2 +\left(\frac{C}{t_0}\right)^2 +C.
  \]
  Hence, at $(x_0,t_0)$,
  \begin{equation}
  \label{eqn-Harnack-3}
  -\dot f\leq \frac{C}{t_0}+\frac{\eta^{i\bar j}f_i f_{\bar j}}{2}+C.
  \end{equation}
  By inequalities \eqref{eqn-Harnack-2} and \eqref{eqn-Harnack-3}, we obtain that
  \[
  \eta^{i\bar j} f_i f_{\bar j}\leq C+\frac{C}{t_0} +\frac{\eta^{i\bar j} f_i f_{\bar j}}{4}
  \]
  i.e.
  \[
  \eta^{i\bar j}f_i f_{\bar j}\leq C+\frac{C}{t_0}.
  \]
  Asserting this inequality to the inequality \eqref{eqn-Harnack-3},
  \[
  -\dot f\leq C+\frac{C}{t_0}.
  \]
  Therefore, we get that
  \[
  \eta^{i\bar j}f_i f_{\bar j} -\alpha \dot f\leq C+\frac{C}{t_0}.
  \]
  Same argument as in the case $\dot f>0$ implies the required inequality.
\end{proof}

As an application of the previous lemma, we derive the following Harnack-type inequality of Li-Yau \cite{LY} in the case of \added{the} line bundle \replaced{MCF}{mean curvature flow}.

\begin{thm}
\label{theorem-5.3}
  There exists constants $C_1$, $C_2$ and $C_3$ such that for all $0<t_1<t_2$, we have the following Harnack-type inequality
  \[
  \sup_{x\in X} v(x,t_1) \leq \inf_{x\in X }  v(x,t_2)\left(\frac{t_2}{t_1}\right)^{C_2} e^{\frac{C_3}{t_2-t_1}+C_1(t_2-t_1)}.
  \]
\end{thm}

\begin{proof}
  Let $x,y\in X$ be two arbitrary points and $\gamma$ be the geodesic with respect to the background metric $\omega$ such that
  \[
  \gamma(0)=x \text{ and }\gamma(1)=y.
  \]
  We also define a curve $\xi(s):[0,1]\to X\times [t_1,t_2]$ by
  \[
  \xi(s)=(\gamma(s),(1-s)t_1+st_2)
  \]
  i.e. $\xi$ is a curve in $X\times [t_1,t_2]$ connecting $(x,t_1)$ and $(y,t_2)$. Then by Lemma \ref{lemma-5.2},
  \[
  \begin{split}
    \ln\frac{v(x,t_1)}{v(y,t_2)}=& -\int_0^1 \frac{\partial }{\partial s} f(\xi(s))ds\\
    =& \int_0^1 (-df(\dot \gamma) -\dot f(t_2-t_1)) ds\\
    \leq & \int_0^1 (\sqrt{\eta^{i\bar j} f_i f_{\bar j}}-\frac{t_2-t_1}{\alpha} \eta^{i\bar j}f_if_{\bar j} -\dot f(t_2-t_1) +\frac{t_2-t_1}{\alpha} \eta^{i\bar j}f_if_{\bar j}) ds\\
    \leq &\int_0^1 \left(-\frac{\alpha}{4(t_2-t_1)} +C(t_2-t_1)+ \frac{C(t_2-t_1)}{(1-s)t_1+ st_2}\right)ds\\
    =&\frac{C_3}{t_2-t_1} +C_1(t_2-t_1) +C_2\ln\frac{t_2}{t_1}
  \end{split}
  \]
  i.e.
  \[
  v(x,t_1)\leq v(y,t_2)\left(\frac{t_2}{t_1}\right)^{C_2} e^{\frac{C_3}{t_2-t_1} +C_1(t_2-t_1)}.
  \]
  Since $x,y$ are arbitrary two points in $X$, we obtain the inequality required.
\end{proof}

\subsection{Exponential\deleted{ly} Convergence}
As a consequence of the Harnack-type inequality above, we first prove the following exponential estimate for
\[
\tilde u= u-\frac{\int_X u\omega^n}{\int_X \omega^n}.
\]

\begin{thm}
\label{thm-exponental-dot}
  There exist two constants $C_1$ and $C_2$ such that
  \[
  \left|\frac{\partial\tilde u}{\partial t}\right|\leq C_1 e^{-C_2 t}.
  \]
\end{thm}

\begin{proof}
  For convenience, we denote $\varphi$ and $\tilde \varphi$ to be $\dot{u}$ and $\dot{\tilde u}$. It is easy to see that $\tilde \varphi$ and $\varphi$ satisfies
  \[
  \int_X\tilde \varphi\omega^n=0 \text{ and }  \frac{\partial \varphi}{\partial t}=\eta^{i\bar j}\varphi_{i\bar j}.
  \]
  Furthermore, for any fixed $t\in [0,\infty)$ and $x,y\in X$, functions $\tilde \varphi$ and $\varphi$ also satisfy the following relation
  \begin{equation}
  \label{eqn-tildephi-phi}
    |\tilde \varphi(x,t)-\tilde \varphi(y,t)|=|\varphi(x,t) -\varphi(y,t)|.
  \end{equation}
  It follows from the maximum principle for the parabolic equation that for any $0<t_1<t_2$, there holds
  \begin{equation}
    \label{eqn-sup}\sup_{y\in X}\varphi(y,t_2)\leq \sup_{y\in X} \varphi(y,t_1)\leq \sup_{y\in X}\varphi(y,0)
    \end{equation}
  and
    \begin{equation}
    \label{eqn-inf}\inf_{y\in X}\varphi(y,t_2)\geq \inf_{y\in X} \varphi(y,t_1)\geq \inf_{y\in X}\varphi(y,0).
  \end{equation}
  Let $m$ be an arbitrary positive integer. For any $(x,t)$, we define
  \[
  \xi_m(x,t)=\sup_{y\in X} \varphi(y,m-1) -\varphi(x,m-1+t)
  \]
  and
  \[
  \psi_m(x,t)=\varphi(x,m-1+t)-\inf_{y\in X} \varphi(y,m-1).
  \]
  Then according to Equations \eqref{eqn-sup} and \eqref{eqn-inf}, $\xi_m$ and $\psi_m$ are both non-negative and satisfy the following parabolic equation
  \[
  \frac{\partial \xi_m}{\partial t}(x,t) =\eta^{i\bar j}(x,m-1+t) (\xi_m)_{i\bar j}(x,t)
  \]
  and
  \[
    \frac{\partial \psi_m}{\partial t}(x,t) =\eta^{i\bar j}(x,m-1+t) (\psi_m)_{i\bar j}(x,t)
  \]
  where $\eta$ depends on the line bundle mean curvature flow $u$.

  In the case that $\varphi(x,m-1)$ is constant, the function $\varphi(x,t)$ must be constant for all $t\geq m-1$ by maximum principle. Hence $\tilde \varphi$ is also a constant for all $t\geq m-1$. But the average of $\tilde \varphi$ vanishes, we obtain that $\tilde \varphi(x,t) =0$ for all $t\geq m-1$. Then our theorem is obvious. Therefore we just need to deal with the case that $\varphi(x,m-1)$ is not constant.

  In the case that $\varphi(x,m-1)$ is not constant, $\xi_m$ must be positive at some point $(x_0,0)$. By \added{the} strong maximum principle, $\xi_m(x,t)$ must be positive for all $x\in X$ when $t>0$. Similarly, we also have $\psi_m(x,t)>0$ for all $x\in X$ when $t>0$. Hence, we can apply Theorem \ref{theorem-5.3} with $t_1=\frac{1}{2}$ and $t_2=1$ to obtain
  \begin{equation}
  \label{eqn-61}
  \begin{split}
    \sup_{y\in X} \varphi(y,m-1) -\inf_{y\in X} \varphi(y,m-\frac{1}{2})&\leq C(\sup_{y\in X} \varphi(y,m-1) -\sup_{y\in X}\varphi(y,m)),\\
    \sup_{y\in X} \varphi(y,m-\frac{1}{2}) -\inf_{y\in X} \varphi(y,m-1)&\leq C(\inf_{y\in X} \varphi(y,m) -\inf_{y\in X} \varphi(y,m-1))
  \end{split}
  \end{equation}
  where $C$ is a positive constant bigger than $1$. We also define $\chi(t)$ to be the oscillation of $\varphi(\cdot,t)$, i.e.
  \begin{equation}
  \label{eqn-62}
  \chi(t)=\sup_{y\in X} \varphi(y,t)- \inf_{y\in X} \varphi(y,t).
  \end{equation}
  Adding the inequalities \eqref{eqn-61} and \eqref{eqn-62} above gives us
  \[
  \chi(m-1)+\chi(m-\frac{1}{2})\leq C(\chi(m-1)-\chi(m)).
  \]
  Since $\chi$ is a non-negative function and $C>1$, there holds
  \[
  \chi(m)\leq \frac{C-1}{C}\chi(m-1).
  \]
  By induction, \deleted{we obtain}
  \begin{equation}
  \label{eqn-chi}
  \chi(m)\leq \left(\frac{C-1}{C}\right)^m \chi(0).
  \end{equation}
  According to the inequality \eqref{eqn-sup} and \eqref{eqn-inf}, we also know that $\chi(t)$ is decreasing in $t$. Therefore, we conclude from \eqref{eqn-chi} that
  \[
  \chi(t)\leq C_1e^{-C_2t}
  \]
  where $C_1=\frac{C\chi(0)}{C-1}$ and $C_2=\ln \frac{C}{C-1}$.

  To obtain the result in the theorem, we observe that there must be a point $x_t\in X$ such that $\tilde \varphi(x_t,t)=0$ for all $t>0$ since $\int_X \tilde \varphi\omega^n=0$. According to Equation \eqref{eqn-tildephi-phi}, for all $(x,t)\in X\times [0,\infty)$,
  \[
  \begin{split}
  |\tilde \varphi(x,t)|=&|\tilde \varphi(x,t)-\tilde \varphi(x_t,t)|\\
  =&|\varphi(x,t)-\varphi(x_t,t)|\leq \chi(t)\leq C_1 e^{-C_2t}
  \end{split}
  \]
  i.e. $\left|\dot{\tilde u}\right|\leq C_1 e^{-C_2t}$.
\end{proof}

We also have the following exponential convergence result for $u$ in $C^\infty$ norm.

\begin{thm}
  The function $\tilde u$ converges exponentially to $0$ smoothly.
\end{thm}

\begin{proof}
  Integrating from $+\infty$ to $t$ and apply Theorem \ref{thm-exponental-dot}, we get that $\tilde u=u-\frac{\int_X u \omega^n}{\int_X \omega^n}$ converges exponentially to $0$ in $C^0$.

  We denote $D$ to be the gradient with respect to $\omega$. For any $k\geq 1$, we consider the following inequality
  \[
  \begin{split}
    &\frac{\partial}{\partial t}\int_X |D^k \tilde u|^2_\omega \omega^n\\
    =&\int_X D^k {\tilde u}* D^k \dot {\tilde u}\omega^n\\
    =&\int_X D^{2k}{\tilde u} *\dot {\tilde u} \omega^n\\
    \leq & (\int_X |D^{2k}{\tilde u}|^2\omega^n)^{\frac{1}{2}}(\int_X {\dot {\tilde u}^2} \omega^n)^{\frac{1}{2}}\\
    \leq &C_1e^{-C_2t}.
  \end{split}
  \]
  Integrating form $+\infty$ to $t$, we get that
  \[
  ||\tilde u||_{W^{k,2}(\omega)}\leq C_1e^{-C_2t}.
  \]
  Then by the Sobolev embedding theorem, we obtain that
  \[
  ||\tilde u||_{C^{k'}}\leq ||\tilde u||_{W^{k,2}(\omega)} \leq C_1e^{-C_2t}.\qedhere
  \]
\end{proof}

\section*{Acknowledgement}

Both authors are grateful to Professor \replaced{Xinan Ma, Xi Zhang and Xiaohua Zhu}{Xiaohua Zhu, Xi Zhang and  Xinan Ma} for helpful suggestions on this subject. The second author is supported by the Fundamental Research Funds for the Central Universities and the Research Funds of Renmin University of China.

\end{document}